\documentclass[12pt]{article}
\usepackage{graphicx,subfigure}
\usepackage{amsmath,amsthm,amssymb,enumerate}
\usepackage{euscript,mathrsfs}
\usepackage{color}
\usepackage{dsfont}
\usepackage{url}
\usepackage{bm}
\usepackage{comment}
\usepackage[left=2cm,right=2cm,top=3.5cm,bottom=3.5cm]{geometry}
\usepackage{color}
\usepackage[framemethod=tikz]{mdframed}
\allowdisplaybreaks

\usepackage{esint}
\usepackage{soul}

\catcode`\@=11 \@addtoreset{equation}{section}

\catcode`\@=12

\newtheorem{Theorem}{Theorem}[section]
\newtheorem{Proposition}[Theorem]{Proposition}
\newtheorem{Lemma}[Theorem]{Lemma}
\newtheorem{Corollary}[Theorem]{Corollary}

\theoremstyle{definition}
\newtheorem{Definition}[Theorem]{Definition}

\newtheorem{Remark}[Theorem]{Remark}

\newcommand{\bTheorem}[1]{
	\begin{Theorem} \label{T#1} }
	\newcommand{\eT}{\end{Theorem}}

\newcommand{\bProposition}[1]{
	\begin{Proposition} \label{P#1}}
	\newcommand{\eP}{\end{Proposition}}

\newcommand{\bLemma}[1]{
	\begin{Lemma} \label{L#1} }
	\newcommand{\eL}{\end{Lemma}}

\newcommand{\bCorollary}[1]{
	\begin{Corollary} \label{C#1} }
	\newcommand{\eC}{\end{Corollary}}

\newcommand{\bRemark}[1]{
	\begin{Remark} \label{R#1} }
	\newcommand{\eR}{\end{Remark}}

\newcommand{\bDefinition}[1]{
	\begin{Definition} \label{D#1} }
	\newcommand{\eD}{\end{Definition}}

\newcommand{\hne}{\vc{h}_{n,\ep}}
\newcommand{\hnet}{\vc{h}_{n,\ep}(t)}
\newcommand{\bap}{\bfphi_{\rm app}^\ep}

\newcommand{\Del}{\Delta_x}

\newcommand{\au}{\vu^\ep_{\rm app}}

\newcommand{\Td}{\mathbb{T}^d}
\newcommand{\Tdd}{\mathbb{T}^3}

\newcommand{\intTd}[1]{\int_{\mathbb{T}^d} #1 \ \dx}
\newcommand{\intTdd}[1]{\int_{\mathbb{T}^3} #1 \ \dx}
\newcommand{\Ds}{\mathbb{D}_x}

\newcommand{\bfomega}{\boldsymbol{\omega}}

\newcommand{\bfphi}{\boldsymbol{\varphi}}

\newcommand{\bfPsi}{\boldsymbol{\Psi}}
\newcommand{\bfpsi}{\boldsymbol{\psi}}

\newcommand{\bFormula}[1]{
	\begin{equation} \label{#1}}
	\newcommand{\eF}{\end{equation}}

\newcommand{\Ov}[1]{\overline{#1}}

\newcommand{\Curl}{{\bf curl}_x}

\newcommand{\aleq}{\stackrel{<}{\sim}}

\newcommand{\vr}{\varrho}
\newcommand{\vre}{\vr_\ep}

\newcommand{\vue}{\vu_\ep}
\newcommand{\tvr}{\wtilde \vr}
\newcommand{\tvu}{{\wtilde u}}

\newcommand{\vu}{\bm{u}}

\newcommand{\vc}[1]{ {\mathbf #1} }
\newcommand{\Ome}{\Omega_\ep}

\newcommand{\Div}{{\rm div}_x}
\newcommand{\Grad}{\nabla_x}

\newcommand{\dx}{\,{\rm d} {x}}

\newcommand{\dt}{\,{\rm d} t }

\newcommand{\intO}[1]{\int_{\Omega} #1 \ \dx}

\newcommand{\intOe}[1]{\int_{\Omega_\ep} #1 \ \dx}

\newcommand{\D}{{\rm d}}

\newcommand{\ep}{\varepsilon}

\newcommand{\R}{\mathbb{R}}

\newcommand{\br}{ \nonumber \\ }

\def\softd{{\leavevmode\setbox1=\hbox{d}%
		\hbox to 1.05\wd1{d\kern-0.4ex{\char039}\hss}}}
\definecolor{Cgrey}{rgb}{0.85,0.85,0.85}
\definecolor{Cblue}{rgb}{0.50,0.85,0.85}
\definecolor{Cred}{rgb}{1,0,0}
\definecolor{fancy}{rgb}{0.10,0.85,0.10}
\definecolor{amaranth}{rgb}{0.9, 0.17, 0.31}

\newcommand\Cbox[2]{%
	\newbox\contentbox%
	\newbox\bkgdbox%
	\setbox\contentbox\hbox to \hsize{%
		\vtop{
			\kern\columnsep
			\hbox to \hsize{%
				\kern\columnsep%
				\advance\hsize by -2\columnsep%
				\setlength{\textwidth}{\hsize}%
				\vbox{
					\parskip=\baselineskip
					\parindent=0bp
					#2
				}%
				\kern\columnsep%
			}%
			\kern\columnsep%
		}%
	}%
	\setbox\bkgdbox\vbox{
		\color{#1}
		\hrule width  \wd\contentbox %
		height \ht\contentbox %
		depth  \dp\contentbox
		\color{black}
	}%
	\wd\bkgdbox=0bp%
	\vbox{\hbox to \hsize{\box\bkgdbox\box\contentbox}}%
	\vskip\baselineskip%
}

\mdfdefinestyle{MyFrame}{%
	linecolor=black,
	outerlinewidth=1pt,
	roundcorner=5pt,
	innertopmargin=\baselineskip,
	innerbottommargin=\baselineskip,
	innerrightmargin=10pt,
	innerleftmargin=10pt,
	backgroundcolor=white!20!white}

\newcommand{\wtilde}{\widetilde}

\allowdisplaybreaks

\begin{document}


\title{\bf ON THE IMPACT\\ OF CLUSTERS OF RIGID BALLS \\ ON THE MOTION OF A VISCOUS FLUID}

\author{Marco Bravin$^{1}$ \thanks{The work of M.B. was supported by the project "Ecuaciones en derivadas parciales motivadas por procesos de difusión y mecánica de fluidos: propiedades, asintóticas y homogeneización" (Ayuda financiada contrato Programa Gob. Cantabria -UC). The work of E.F. was partially supported by the
		Czech Sciences Foundation (GA\v CR), Grant Agreement
		24--11034S. The Institute of Mathematics of the Academy of Sciences of
		the Czech Republic is supported by RVO:67985840.
		E.F. is a member of the Ne\v cas Center for Mathematical Modelling. A.R is supported by the Grant RYC2022-036183-I funded by MICIU/AEI/10.13039/501100011033 and by ESF+. A.R, A.Z have been partially supported by the Basque Government through the BERC 2022-2025 program and by the Spanish State Research Agency through BCAM Severo Ochoa CEX2021-001142-S and through project PID2023-146764NB-I00 funded by MICIU/AEI/10.13039/501100011033 and cofunded by the European Union. A.Z. was also partially supported  by a grant of the Ministry of Research, Innovation and Digitization, CNCS - UEFISCDI, project number PN-III-P4-PCE-2021-0921, within PNCDI III.} 
	\and Eduard Feireisl$^{2}$
	\and Arnab Roy$^{3,4}$ \and Arghir Zarnescu$^{3,4,5}$
}

\date{}

\maketitle

\centerline{$^{1}$Universidad de Cantabria, E.T.S. de Ingenieros Industriales y de Telecomunicacíon}

\centerline{Departmento de Matemática Aplicada y Ciencias de la Computacíon}

\centerline{Avd. Los Castros 44, 39005 Santander, Spain}

\bigskip

\centerline{$^2$ Institute of Mathematics of the Academy of Sciences of the Czech Republic}

\centerline{\v Zitn\' a 25, CZ-115 67 Praha 1, Czech Republic}

\bigskip
\centerline{$^3$ BCAM, Basque Center for Applied Mathematics}

\centerline{Mazarredo 14, E48009 Bilbao, Bizkaia, Spain.}
\bigskip
\centerline{$^4$IKERBASQUE, Basque Foundation for Science, }

\centerline{Plaza Euskadi 5, 48009 Bilbao, Bizkaia, Spain.}
\bigskip
\centerline{$^5$``Simion Stoilow" Institute of the Romanian Academy,}

\centerline{21 Calea Grivi\c{t}ei, 010702 Bucharest, Romania.}

\medskip

\date{}

\begin{abstract}
	
We develop a new approach to the problem of the motion of a large number of rigid bodies immersed in a viscous fluid. The leading idea is the concept 
of \emph{cluster} - a collection of individual rigid objects that 
may be grouped or even connected in such a way that their 
collective impact on the bulk motion of the system is similar to that of a single body. The applications of the new approach include: 

\noindent
$\bullet$ Improving the critical value of the number of balls 
of small radius such that their cloud has no impact on the limit system represented by the incompressible Navier--Stokes equations.

\noindent
$\bullet$ 
The balls follow the fluid flow in the asymptotic limit
of vanishing radius and increasing number 
even if a gravitational force is imposed.

\end{abstract}


{\small

\noindent
{\bf 2020 Mathematics Subject Classification:} 

\medbreak


}

\section{Introduction}
\label{i}

We consider a cloud of $N$ small rigid balls $B(r, \vc{h_n})
= \{x; |x - \vc{h}_n| < r \}$, $n=1, \dots, N$ of a common radius $r$ immersed in a viscous fluid occupying a domain in $\Omega \subset \R^d$, $d=2,3$. Their positions at a given time $t \geq 0$ are determined by a family 
of affine isometries $\sigma_1(t), \dots, \sigma_N(t)$ such that:

\begin{equation}
B(r, \vc{h}_n(t)) = \sigma_n(t) \{ B(r, 0) \},\ \sigma_n(t)x = \vc{h}_n (t) + 
\mathbb{O}_n(t)x ,\ \mathbb{O}_n(t) \in SO(d),\ n = 1,\dots, N,  
\label{i1}
\end{equation}
with the associated rigid velocities
\begin{equation} \label{i2} 
\vu_{n}^S(t,x) = \frac{\D \vc{h}_n (t) }{\dt}  + \bfomega_n (t) \times (x - \vc{h}_n(t)),\ 
\bfomega_n(t) = \frac{\D }{\dt} \mathbb{O}_n \circ \mathbb{O}^{-1}_n(t),\ 
n = 1, \dots N. 
\end{equation}

The fluid is Newtonian of a constant density $\vr_F > 0$, and the time evolution of its velocity $\vu^F = \vu^F(t,x)$ and the pressure $p = p(t,x)$  is governed by the Navier--Stokes system of equations 
\begin{align} 
	\Div \vu^F &= 0, \br
\vr_F \Big( \partial_t \vu^F + \vu^F \cdot \Grad \vu^F \Big) + \Grad p &= 
\Div \mathbb{S} + \vr^F \vc{g} \label{i3}
\end{align}	
with the viscous stress 
\begin{equation} \label{i4}
\mathbb{S} = \mathbb{S}(\Ds \vu^F) = \nu \Big(\Grad \vu^F + \Grad^t \vu^F \Big),\ 
\nu > 0,\ \Ds \vu = \frac{1}{2}	\Big(\Grad \vu^F + \Grad^t \vu^F \Big),\ \nu > 0.
\end{equation} 
The symbol $\vc{g} = \vc{g}(t,x)$ denotes a given volume force imposed on the 
fluid - balls system, typically a gravitational force $\vc{g} = \Grad G$.

The equations for $\vu^F,p$ hold at any given time $t>0$ in the fluid domain $\Omega\setminus\cup_{i=1}^N B(r,h_n(t))$.
We consider the no--slip boundary conditions at the fluid--ball interface, 
\begin{equation} \label{i5}
	\vu^F|_{\partial B(r, \vc{h}_n)} = \vu^S_n,\ n = 1, \dots, N.  
\end{equation}
In addition, 
we suppose the balls have constant mass densities $0 < \vr_S^n < \Ov{\vr}$, 
$n=1, \dots, N$, and their motion is governed by the  
momentum equations, which  for $d=3$ read as: 
\begin{align} 
	m_n \frac{\D^2 }{\dt^2} \vc{h}_n(t) &= \int_{\partial B(r, \vc{h}_n(t))} \Big( 
	\mathbb{S} - p \mathbb{I} \Big) \cdot \vc{n}\, \D S_x + \int_{B(r, \vc{h}_n(t))} 
	\vr^n_S \vc{g}\ \dx,\ m_n = \vr^n_S |B(r,0)| ,\br
\mathbb{J}_n(t) \frac{\D }{\dt} \bfomega_n(t) &= \mathbb{J}_n(t) \bfomega_n(t) \times \omega_n(t) \br &+ \int_{\partial B(r, \vc{h}_n(t)) } \Big( x - \vc{h}_n(t) \Big) \times \Big( \mathbb{S} - p \mathbb{I} \Big) \vc{n}   
\D S_x + \int_{B(r, \vc{h}_n(t))} \vr^n_S (x - \vc{h}_n(t)) \times \vc{g} \dx,
\label{i6}
\end{align}  with the $3\times 3$ matrix $\mathbb{J}_n $ defined through its action on arbitrary vectors $a, b\in\R^3$

\begin{equation*}
	\mathbb{J}_n (t) \vc{a} \cdot \vc{b} = 
\int_{B(r, \vc{h}_n(t))} \vr^n_S \left[ \vc{a} \times \Big( x - \vc{h}_n(t) \Big) \right] \cdot \left[ \vc{b} \times \Big( x - \vc{h}_n(t) \Big) \right]\dx.
\end{equation*}

The system of equations \eqref{i1}--\eqref{i6} represents a strong formulation of \emph{fluid--structure interaction problem}. The two-dimensional counterpart of the system \eqref{i1}--\eqref{i6} can be found in \cite[pp.~14--15]{FeiRoyZar2023}.

\subsection{Dynamic homogenization}
\label{ip1}

Our main objective is to study the asymptotic behaviour of solutions of the  fluid--structure interaction problem \eqref{i1}--\eqref{i6} in the regime $r \to 0$, $N \to \infty$, and  possibly when $\Ov{\vr} \to \infty$. The problem can be interpreted as a dynamic 
version of the homogenization process, where the fluid domain is perforated by a family o fixed small holes, see the pioneering work of Allaire \cite{Allai4}, 
\cite{Allai3} or 
Mikeli{\v c} \cite{Mik} among many others. As we have shown in \cite{BrFeRoZa}, the two processes are close and give rise to similar results 
on condition that the density of the rigid balls $\vr^n_S$ is extremely large and becomes infinite in the asymptotic limit. Intuitively, heavy bodies do not move and produce a similar asymptotic effect as fixed holes pasted in the fluid domain. 

A more interesting and physically relevant situation arises if the densities 
$\vr^n_S$ stay bounded or not very large in the asymptotic limit $r \to 0$, $N \to \infty$.
Despite a large number of studies devoted to the motion of one or a finite number of rigid bodies, mostly up to the ``first contact'', see 
\cite{BraNec2}, \cite{BraNec1}, He and Iftimie \cite{HeIft1}, \cite{HeIft2}, 
Lacave and Takahashi \cite{LacTak}, results concerning the asymptotic limit for a large number of rigid bodies are in a short supply. To the best of our knowledge, the problem of ``dynamic homogenization'' was first addressed in  
\cite{FeiRoyZar2022}. It was shown that given the radius $r \to 0$, there is a critical number $N = N(r)$ proportional to $\log(r)$ such that the cloud of balls has no influence on the asymptotic limit that coincides with the Navier--Stokes system. In particular, this number is higher than in the corresponding homogenization limit in identified by 
Allaire \cite{Allai3} in the dimension $d=2$. What is more, 
the result still holds in the context of non-Newtonian fluids examined in 
\cite{FeiRoyZar2025}. This strongly suggests that, 
unlike in the standard (static) homogenization process, the total \emph{capacity} of the family of balls is not essential for the asymptotic limit.

The fact that the balls immersed in a fluid are allowed to follow the fluid motion, pertinent to dynamic homogenization, gives rise to two hypothetical rather complementary situations:

\begin{enumerate}
\item The balls stay far away one from another. The dynamics is therefore similar that of a single body.

\item The balls stay together, possibly collide, and form more complex objects. Still the more complex objects, remaining together, impose the same 
effect on the fluid motion as a single body.

\end{enumerate}

The approach proposed in \cite{FeiRoyZar2022} is tailored  to handle the second scenario. In addition, the fact that the balls that remain far away from collisions have less influence on the development of the system is ignored. In the present paper, we propose a completely new approach to the same problem based on the concept of \emph{clusters}. 
A cluster of rigid balls is a collection of individual rigid objects that 
may be grouped or even connected in such a way that their 
collective impact on the bulk motion of the system is similar to that of a single body. We take advantage of this new approach to improve the existing results in two directions:

\begin{itemize}
	
\item We considerably improve the critical value of the number of balls 
of radius $r$, namely $N \approx r^{- \beta}$ for some $\beta > 0$, such that their cloud has no impact on the limit system formed by the incompressible Navier--Stokes equations. It is worth noting 
the result holds also in the case $d=2$. 
In particular, the number of balls not influencing the fluid motion 
is of orders of magnitude higher than  
in the homogenization process, cf. Allaire \cite{Allai4}. 

\item We show that balls follow the fluid flow in the asymptotic limit $r \to 0$, $N \to \infty$ even if the mass density $\vr^n_S$ 
is moderately growing. To this end, we adapt the method of relative energy to compare the distance between 
the ball velocities and that of the fluid. To the best of our knowledge, this is the first result concerning the rate of convergence in the dynamic homogenization limit. 
	
\end{itemize}	

We point out that our results are unconditional and global in time. In particular, collisions of two and more rigid balls are allowed. As a matter of fact, the method of clusters is intended to handle this problem intrinsic to dynamic homogenization, where the mutual distance of rigid objects tends to zero. 

\subsection{Organization of the paper}

Our results are stated in the framework of weak solutions introduced by Judakov \cite{Juda}.
The definition of weak solutions and the main results for the fluid--structure interaction problem are given in Section \ref{m}.
In Section \ref{p}, we introduce the spatial domain considered and show how the problem can be transformed to a purely space perodic setting. Section \ref{C} is the heart of the paper. We introduce the concept of cluster and cluster distance and cluster projection, together with their basic properties. In particular, in Section \ref{F}, we show how the cluster theory can be applied the the fluid--structure interaction problem. In Section \ref{A}, we construct a suitable approximation of smooth test functions for the momentum equation. Finally, we complete the proof of our main result in Section \ref{w} where we prove convergence to a weak solutions, respectively in Section \ref{T} where we prove convergence to a strong solution, which allows to obtain rates of convergence and a characterization of the limit dynamics of the rigid balls.

\section{Main results}
\label{m}

In the section, we introduce the concept of weak solution to the fluid--structure interaction problem and state our main results.

\subsection{Geometry of the physical space}

To avoid technicalities related to the boundary conditions, we suppose the fluid containing the balls is confined in a periodic slab 
\begin{equation} \label{m1}
	\Omega = \left\{ (\vc{x}_h, x_3) \ \Big| \ \vc{x}_h    \in \mathbb{T}^{d-1},\ x_d \in (0,H) \right\} = \mathbb{T}^2 \times (0,H),
\end{equation}	
where the symbol 
\[
\Td = \left( [-L, L]|_{-L, L} \right)^d 
\]
denotes the flat torus in $\R^d$. The horizontal boundary $x_d = 0, H$ is impermeable, meaning the fluid velocity satisfies
\begin{equation} \label{m2}
\vu^F \cdot \vc{n} = u^F_d|_{x_d = 0,H} = 0, \ 
[\mathbb{S} \cdot \vc{n}]_{\rm tan}|_{x_d = 0,H} = 0.
\end{equation}	 where $[S\cdot n]_{tan}$ denotes the tangential component of $[S\cdot n]$, namely $[S\cdot n]_{tan}=n\times ([S\cdot n]\times n)$.

\subsection{Weak solutions}

We are in a position to introduce the weak formulation of the fluid--structure interaction problem: 

\begin{Definition} [\bf Weak solution] \label{Dm1}
	
The quantity $(\vr, \vu)$, together with $(\vc{h}_1, \dots, \vc{h}_N, \mathbb{O}_1, \dots, 
\mathbb{O}_N)$ is \emph{weak solution} of the fluid--structure interaction problem \eqref{i1}--\eqref{i6}, \eqref{m2} in $(0,T) \times \Omega$, 
$\Omega$ given by \eqref{m1}, with the initial data 
\begin{equation} \label{m3}
	\vr(0, \cdot) = \vr_0,\ \vr \vu(0, \cdot) = \vr_0 \vu_0,
\end{equation}
if the following holds:
\begin{itemize}
\item {\bf Regularity.}
\begin{align} 
	\vr &\in L^\infty((0,T) \times \Omega) \cap C([0,T]; L^1(\Omega)),\ 
	\vu \in L^\infty(0,T; L^2(\Omega; \R^d)) \cap 
	L^2(0,T; W^{1,2}(\Omega; \R^d)),\ \br
	\vc{h}_n &\in W^{1,\infty}(0,T; \R^d),\ \mathbb{O}_n \in W^{1,\infty}(0,T; SO(d)),\ n = 1, \dots, N.
\label{m4}
\end{align}
\item {\bf Compatibility.}
\begin{align} 
\vr(t, \cdot)|_{B(r, \vc{h}_n(t))} &= \vr^n_S,\ n = 1, \dots, N,\ 
\vr(t, \cdot) = \vr_F\ \mbox{in}\ \Omega \setminus \cup_{n=1}^N B(r, \vc{h}_n(t)),\ t \in [0,T], \br 
\vu(t, \cdot)|_{B(r, \vc{h}_n(t))} &= \vu^S_n (t, \cdot), \ n=1,\dots, N,\ \vu \cdot \vc{n} = u_d|_{x_d = 0,H} = 0,
\label{m5}	
\end{align}	
where $\vc{u}^S_n$ are given by \eqref{i2}.
\item {\bf Mass conservation.} 
\begin{equation} \label{m6}
\int_0^T \intO{ \Big[ \vr \partial_t \varphi + 
	\vr \vu \cdot \Grad \varphi \Big] } \dt = - \intO{ \vr_0 \varphi(0, \cdot) },\ \Div \vu = 0
\end{equation}
for any test function $\varphi \in C^1_c([0,T) \times \Ov{\Omega})$. 	
\item {\bf Balance of momentum.}
\begin{align} \label{m7}
\int_0^T &\intO{ \Big[ \vr \vu \cdot \partial_t \bfphi + 
\vr (\vu \otimes \vu) : \Grad \bfphi \Big] } \dt \br 
&= \int_0^T \intO{ \Big[ \mathbb{S} (\Ds \vu) : \Grad \bfphi - \vr \vc{g} \cdot \bfphi \Big] } \dt - \intO{ \vr_0  \vu_0 \cdot \bfphi(0, \cdot)}
\end{align}	
for any test function $\bfphi \in C^1_c([0,T) \times \Ov{\Omega}; \R^d)$, $\bfphi \cdot \vc{n}|_{\partial \Omega} = 0$, 
$\Div \bfphi = 0$ satisfying 
\begin{equation} \label{m8}
\Ds \bfphi (t, \cdot) = 0 \ \mbox{on some open neighbourhood of}\ 
\cup_{n=1}^N B(r, \vc{h}_n(t)) \ \mbox{in}\ \Ov{\Omega} 
\ \mbox{for any}\ t \in [0,T].
\end{equation}	
\item {\bf Energy inequality.}
\begin{align} \label{m9}
- \int_0^T \partial_t \psi &\intO{  \frac{1}{2} \vr |\vu|^2 } + 
\int_0^T \psi \intO{ \mathbb{S} (\Ds \vu) : \Grad \vu } \dt \br &\leq \psi(0) 
\intO{ \vr_0 |\vu_0|^2 } + \int_0^T \psi \intO{ \vr \vc{g} \cdot \vu } \dt  
\end{align}	
for any $\psi \in C^1_c[0,T)$, $\psi \geq 0$.
\end{itemize}

\end{Definition}	

The concept of weak solution goes back to Judakov \cite{Juda}, see also 
Galdi \cite{GAL1}, Gunzburger, Lee, and Seregin \cite{GLSE} among others. 
In particular, hypothetical contacts of two or more rigid bodies are allowed, whereas the bodies can either separate or stay together after collisions. 

As we shall see in the next section, problem \eqref{m1}--\eqref{m2} can be reformulated as purely periodic with respect to all space variables. In such a setting, the \emph{existence} of global in time weak solutions was proved in \cite{BrFeRoZa2025}.

\subsection{Main results}

Our first result establishes a bound on the number of balls 
in terms of their common radius such that the limit 
system does not feel their collective effect.

\begin{Theorem} [\bf Asymptotic limit] \label{Tm1}
	
Let $\Big(\vre, \vue, (\hne )_{n=1}^{N_\ep},\ 
	(\mathbb{O}_{n, \ep})_{n=1}^{N_\ep} \Big)_{\ep > 0}$ be a family of weak solutions to the fluid structure interaction problem \eqref{m3}--\eqref{m9}, 
with the initial data 
\begin{equation} \label{m10}
|\vc{h}_{i, \ep}(0) - \vc{h}^0_{j, \ep}(0)| \geq 2 r_\ep \ \mbox{for}\ i \ne j,\ 	
	\sqrt{\vr_{0,\ep}} \vu_{0, \ep} \to \sqrt{\vr_F} \vu_0 
	\ \mbox{weakly in}\ L^2(\Omega; \R^d).
\end{equation}		
where 
\begin{equation} \label{m11}
r = r_\ep,\ \vc{g} = (0,0, g) \in L^\infty((0,T) \times \Omega),\ 
0 < \vr^{n,\ep}_S \leq \Ov{\vr}_\ep, \ n = 1, \dots, N(\ep)\mbox{ uniformly for  }\ep \to 0. 
\end{equation}
Suppose
{
\begin{align} 
N_\ep\bigg(|\log(r_\ep)|^{\frac{4}{5}}  r_\ep^{\frac{d}{5}} &+ \Ov{\vr}_\ep^{q} |\log(r_\ep)|^{\frac{6}{11}} r^{\frac{12}{11}}_\ep \bigg)\to 0 \ \mbox{as}\ \ep \to 0 \ \textrm{ for some}\ q > \frac{6}{11}\  \mbox{if}\ d = 3, \br 
N_\ep\bigg(|\log(r_\ep)|^{\frac{4}{5}}  r_\ep^{\frac{d}{5}} &+ \Ov{\vr}_\ep^{\frac{q}{q+1}} |\log(r_\ep)|^{\frac{q}{q+1}} r^{\frac{1}{q+1}}_\ep\bigg) \to 0 \ \mbox{as}\ \ep \to 0  
\ \mbox{for some}\ q > 1\	 \mbox{if}\ d = 2.
\label{m12}
\end{align}	
}

Then there exists a subsequence (not relabelled for simplicity) such that 
\begin{align} 
	\vr_\ep &\to \vr_F \ \mbox{in}\ C([0,T]; L^\gamma(\Td, \R^d)),{\mbox{ for some } }\gamma > \frac{d}{2}, \br 
	\vue &\to \vu \ \mbox{weakly-(*) in}\ L^\infty(0,T; L^2(\Omega; \R^d)), 
	\ \mbox{weakly in} \ L^2(0,T; W^{1,2}(\Omega; R^3)), \br 
	&\quad \quad \mbox{\ and (strongly) in}\ L^2((0,T) \times \Omega; \R^d),  
\label{m13}	
\end{align}
where $(\vr_F, \vu)$ is a weak solution to the Navier--Stokes system
\begin{align} 
\Div \vu &= 0 \ \mbox{a.a. in}\ (0,T) \times \Omega, \br 
\int_0^T &\intO{ \Big[ \vr \vu \cdot \partial_t \bfphi + 
		\vr (\vu \otimes \vu) : \Grad \bfphi \Big] } \dt \br 
	&= \int_0^T \intO{ \Big[ \mathbb{S}: \Grad \bfphi - \vr \vc{g} \cdot \bfphi \Big] } \dt - \intO{ \vr_0\vu_0 \cdot \bfphi(0, \cdot)}
	\label{m14}	  		
\end{align}	 
for any $\bfphi \in C^1_c([0,T) \times \Ov{\Omega}; \R^d)$, 
$\Div \bfphi = 0$, $\bfphi \cdot \vc{n}|_{\partial \Omega} = 0$.
\end{Theorem}

The result improves considerably the estimate on $N(\ep) \aleq - \log(\ep)$ obtained in \cite{FeiRoyZar2022}, and largely exceeds the 
critical number of balls in the ``static'' homogenization problem for $d=2$.

Our next result concerns the asymptotic behaviour of the trajectories of individual rigid balls in the asymptotic limit $\ep \to 0$ provided the background Navier--Stokes system admits a smooth solution.
	
\begin{Theorem} [\bf Convergence to smooth limit] \label{Tm3} 
Let $\Big(\vre, \vue, (\hne )_{n=1}^{N_\ep},\ 
(\mathbb{O}_{n, \ep})_{n=1}^{N_\ep} \Big)_{\ep > 0}$ be a family of weak solutions to the fluid structure interaction problem \eqref{m3}--\eqref{m9}, 
with the initial data 
\begin{align} 
{|\vc{h}_{i, \ep}(0) - \vc{h}_{j, \ep}(0)|} &\geq 2 r_\ep \ \mbox{for}\ i \ne j,\ \vc{h}_{n,\ep}(0) \to \vc{h}_{n, 0}\ \mbox{as}\ \ep \to 0  
\ \mbox{uniformly for}\ n = 1, \dots, N_\ep, \br	
	\intTd{ \vre |\vu_{0,\ep} - \vu_0 |^2 } &\leq r^\alpha_\ep \beta (\ep),\ 
	\beta (\ep) \to 0 \ \mbox{as}\ \ep \to 0,\ \alpha = 1 \ \mbox{if}\ d = 3,\ \alpha > 0 \ \mbox{arbitrary if}\ d = 2.
\label{m21}	
\end{align}	
In addition, suppose 
\begin{equation} \label{m11R}
	r = r_\ep,\ \vc{g} = (0,0, g) \in L^\infty((0,T) \times \Omega),\ 
	0 < \underline{\vr} \leq \vr^{n,\ep}_S \leq \Ov{\vr}_\ep, \ n = 1, \dots, N_\ep\ \mbox{uniformly for} \ \ep \to 0, 
\end{equation}
where 
\begin{align} \label{m20} 
&N_\ep\bigg( |\log(r_\ep)|^{\frac{4}{5}} r_\ep^{\frac{d-1}{5}} + \Ov{\vr}_\ep^{\frac{6}{11}}  |\log(r_\ep)|^{\frac{6}{11}} r^{\frac{9}{11}}_\ep\bigg) \to 0 \ \mbox{as}\ \ep \to 0 \ \mbox{if}\ d = 3
\br
&N_\ep \bigg(|\log(r_\ep)|^{\frac{4}{5}} r_{\ep}^{\frac{d-\delta}{5}} + \Ov{\vr}_\ep^{\frac{q}{q+1}} |\log(r_\ep)|^{\frac{q}{q+1}} r^{\frac{1}{q+1}(1-\delta)}_\ep \bigg)\to 0 \ \mbox{as}\ \ep \to 0 
  \mbox{ for some}\ q > 1, \delta>0 \ \mbox{if}\ d = 2. 
\end{align}
Finally, suppose that
the Navier--Stokes system \eqref{m14} admits a classical solution $\vu$ in $(0,T) \times \Omega$, in particular 
with $\partial_t \vu$, $\nabla^2_x \vu$ continuous in $[0,T] \times \Ov{\Omega}$.

Then 
	\begin{align} 
		\vr_\ep &\to \vr_F \ \mbox{in}\ C([0,T]; L^\gamma(\Td, \R^d)),{\textrm{ for some }}\ \gamma > \frac{d}{2}, \br 
		\vue &\to \vu \ \mbox{in}\ L^2(0,T; W^{1,2}(\Omega; \R^d)), \br 
		\sqrt{\vre} \vue &\to \sqrt{\vr_F} \vu \ \mbox{in}\ L^\infty(0,T; L^2(\Td; \R^d)), 
		\label{m22}	
	\end{align}
    
and 
\begin{align} 
\vc{h}_{n, \ep} &\to \vc{h}_n \ \mbox{in}\ W^{1,2}(0,T; \R^d) \ \mbox{as}\ \ep \to 0, \br
\mbox{where}\ \frac{\D }{\dt} \vc{h}_n(t, \cdot) &= \vu (t, \vc{h}_n(t))
\ \mbox{for any}\ n = 1,2,\dots
\label{m23}
\end{align}
\end{Theorem}

As is well known, the limit Navier--Stokes system admits a local-in-time classical solution 
as soon as the initial data are sufficiently smooth. The solution is global if $d=2$,

\subsection{Problem of sedimentation}

We conclude this introductory part by a short discussion comparing our results with the work of 
H\" ofer and Schubert \cite{HofSchuI} and Jabin and Otto \cite{JabOtt} on the problem of sedimentation. 
For simplicity, we focus on $d=3$. We consider a cloud of periodically distributed equal balls driven by a vertical gravitational field. Specifically, we set 
\begin{align} \label{m24}
\ep = \frac{L}{2K + 1},\ \vc{h}_{2k,2\ell}(0) &= \Big[{2 k \ep   , 2\ell  \ep }   , \frac{H}{2} \Big] \in \Omega,\ r_{k, \ell} = \lambda \ep,\ 0 < \lambda < 1,\ 
k = -K , \dots, K,\ \ell = -K, \dots, K,\br 
\vu_{0} &= 0, \ \vr_{S}^{k, \ell} = \vr_S > 0,\ \vc{g} = (0,0,-1) = - \Grad x_3.
\end{align}
Thus there are $N = (2K + 1)^2 = \left( \frac{L}{\ep} \right)^2$ balls of radius 
$\lambda \ep$ periodically distributed in the domain $\Omega$.

Consider the fluid--structure interaction problem \eqref{m3}--\eqref{m9} 
with the data \eqref{m24}:

\begin{enumerate}
\item The first, very simple but crucial observation is that due to the periodicity 
of the domain $\Omega$, the solution can be obtained by a simple spatial shift
of a solution $(\tvr, \tvu)$ 
of the one body problem
on the (periodic) strip:
\[
\Omega_\ep = \left( [- \ep, \ep]|_{\left\{ - \ep , \ep \right\} } \right)^2 \times [0,H], 
\]
with a single ball immersed in the fluid centred at 
\[
\vc{h}(0) = \vc{h}_{0,0}(0) = (0,0, \frac{H}{2}),\ r = \lambda \ep,\ 
\tvu(0, \cdot) = \tvu_0 = 0.
\]
Indeed the solution $(\tvr, \tvu)$ extended periodically outside $\Omega_\ep$, 
\begin{align} 
	\vr(t, x_1 + 2 \ep k, x_2 + 2 \ep \ell, x_3) &= \tvr(t, x_1, x_2, x_3),\br 
	{\vu(t, x_1 + 2\ep k, x_2 + 2\ep \ell, x_3)} &= \tvu(t, x_1, x_2, x_3),\br
t \in (0,T),\ x \in \Omega_\ep ,\ k &= -K, \dots K,\ \ell = -K, \dots, K
\label{m25}	 
\end{align}
yields a weak solution$(\vr, \vu)$ of \eqref{m3}--\eqref{m9} in $(0,T) \times \Omega$.

\item As for the solution $(\tvr, \tvu)$, we rewrite the energy inequality 
\eqref{m9} in the form 
\begin{equation} \label{m26}
\frac{\D }{\dt} \intOe{ \left[ \frac{1}{2} \tvr |\tvu|^2 + x_3 \tvr \right]     }
+ \nu \intOe{ |\Grad \tvu|^2   } \leq 0 \ \mbox{in}\ \mathcal{D}'(0,T), 
\end{equation}	
cf. \cite{BrFeRoZa2025}. 

Next, we use a version of Poincar\' e inequality proved in Section \ref{p} below, 
\begin{equation} \label{m27}
\left\| v - \left( \intOe{\tvr} \right)^{-1} \intOe{ \tvr v } \right\|^2_{L^2(\Ome)} \leq C_P \| \Grad v \|^2_{L^2(\Ome; R^{3})} 
\end{equation}	
where the constant $C_P$ depends only on 
\[
\intOe{ \tvr } \ \mbox{and}\ \sup_{\Ome} \tvr .
\]	

As $\tvr, \tvu$ satisfy the momentum balance \eqref{m14} and $\tvu_0 = 0$, we easily deduce 
\[
\intOe{ \tvr \widetilde{u}_i (t, \cdot) } = 0 ,\ i = 1,2, {\textrm{ for all }t\in [0,T]}
\]
while $\widetilde{u}_3|_{\partial \Ome} = 0$.
Going back to \eqref{m26} we conclude 
\begin{equation} \label{m28}
	\frac{\D }{\dt} \intOe{ \left[ \frac{1}{2} \tvr |\tvu|^2 + x_3 \tvr \right]     }
	+ \Lambda \intOe{ \left[ |\Grad \tvu|^2 + |\tvu|^2 \right]   } \leq 0 \ \mbox{in}\ \mathcal{D}'(0,T)
\end{equation}	 
for some $\Lambda > 0$. Following the arguments of \cite{BrFeRoZa2025} we may infer that 
\begin{align}
\intOe{ \tvr |\tvu|^2 (t, \cdot) } &\to 0 \ \mbox{as}\ t \to \infty, \br 	
\intOe{ x_3 \tvr (t, \cdot) } &\to a \ \mbox{as}\ t \to \infty \ \Rightarrow \ 
(\vr_S - \vr_F) \int_{B(t)} x_3 \ \dx \to b \ \mbox{as}\ t \to \infty. 
\label{m29}
\end{align}{ with $a$ and $b$ some real constants.}

\item Suppose $\vr_S \ne \vr_F$. Seeing that 
\[
h^3(t) = \frac{1}{|B(\lambda \ep, \vc{h}(t))|} \int_{B(\lambda \ep, \vc{h}(t))} x_3 \dx, 
\]
we deduce from \eqref{m29} that 
\begin{equation} \label{m30}
	h^3(t) \to h_\infty \ \mbox{as}\ t \to \infty.
\end{equation}	
Finally, using the argument of \cite{FeiNec2011}, together with 
the hypothesis $\tvu_0 = 0$, 
we conclude
\begin{align} 
h^3(t) &\to r = \lambda \ep \ \mbox{as}\ t \to \infty,\ \mbox{if}\ \vr_S > \vr_F, \br 	
h^3(t) &\to H - r = H - \lambda \ep \ \mbox{as}\ t \to \infty \ \mbox{if}\ \vr_S < \vr_F.
\label{m31}
\end{align}

\end{enumerate}

Let us summarize the previous discussion. 

\begin{Theorem}[\bf Sedimentation] \label{Tm4}
Let the initial data for the fluid--structure interaction problem \eqref{m1}--\eqref{m9} be given by \eqref{m24}. 

Then there exists a weak solution $(\vr, \vu)$ enjoying the following properties:
\begin{itemize}
\item The solution $(\vr, \vu)$ is global in time, there is no collision of two or more rigid balls at a finite time. 

\item The kinetic energy of the system satisfies 
\[
\intO{ \frac{1}{2} \vr |\vu|^2 (t ,\cdot) } \to 0 
\ \mbox{as}\ t \to \infty.
\]
\item 
\[
h^3_{k, \ell} (t) = h^3_{0,0} (t) 
\ \mbox{for all}\ k = - K, \dots, K,\ \ell = -K, \dots, K.
\] 
	
\item 	
\begin{align} 
	h^3_{k, \ell}(t) &\to r = \lambda \ep \ \mbox{as}\ t \to \infty,\ \mbox{if}\ \vr_S > \vr_F, \br 	
	h^3_{k, \ell} (t) &\to H - r = H - \lambda \ep \ \mbox{as}\ t \to \infty \ \mbox{if}\ \vr_S < \vr_F,\ t > 0.
	\nonumber
\end{align}

\end{itemize}	
\end{Theorem}	

Let us recall that the number of rigid balls in Theorem \ref{Tm4} 
is $N = \frac{L^2}{\ep^2}$. We conjecture such a hypothesis would be optimal the framework of Theorem \ref{Tm1}. In contrast with the work of Otto and Jabin \cite{JabOtt}, the fluid domain is a bounded strip. Consequently, the velocities of the all rigid bodies go to zero uniformly with respect to $\ep$. 


\section{Transformation to a purely space periodic problem}
\label{p}

To avoid technicalities connected with the presence of a physical boundary, we adapt the approach proposed by Ebin \cite{EB}. 

First, we extend/replace the periodic strip $\Omega$ by  
\[
\mathbb{T}^d = \left( [-L,L]|_{\{ - L, L \} } \right)^{d-1} \times 
[-H, H] |_{ \{ -H, H \} }.
\]

Next, we extend the initial data to $\mathbb{T}^d$ so that 
\begin{align} 
\vr_0 (\vc{x}_h, - x_d) &= \vr_0( \vc{x}_h, x_d), \br
u_0^i (\vc{x}_h, - x_d) &= u_0^i( \vc{x}_h, x_d), \ i \leq d - 1, \br	
u^d_0 (\vc{x}_h, - x_d) &= - u^d_0 ( \vc{x}_h, x_d),	\label{p1}
\end{align}	
with a symmetric distribution of the rigid ball with respect to 
the plane $x_d = 0$. 

Similarly, the volume force $\vc{g} = (0,0, g)$ is extended as 
\begin{equation} \label{p2}
g(\vc{x}_h, - x_d) = - g(\vc{x}_h, x_d).
\end{equation}

Finally, we suppose that the weak solution 
of the fluid--structure interaction problem on the flat torus 
$\Td$ in the sense of Definition 
\ref{Dm1} belongs to the same symmetry class, 
\begin{align} 
	\vr (t, \vc{x}_h, - x_d) &= \vr(t, \vc{x}_h, x_d), \br
	u^i (t, \vc{x}_h, - x_d) &= u^i(t, \vc{x}_h, x_d), \ i \leq d -1, \br	
	u^d (t, \vc{x}_h, - x_d) &= - u^d (t, \vc{x}_h,  x_d),\ t \in (0,T).	\label{p3}
\end{align}	
Then, using the periodicity, it is straightforward to verify that the solution \eqref{p3}, when restricted to the strip $\Omega$, satisfies the required complete slip boundary conditions \eqref{m2}.

Solutions belonging to the symmetry class \eqref{p3} can be constructed exactly as in \cite{BrFeRoZa2025}. Throughout 
the remaining text, we shall therefore always tacitly assume that 
the weak solution are defined on the periodic (flat) torus 
$\mathbb{T}^d$ and belong to the symmetry class \eqref{p3}. 

\subsection{Poincar\'e inequality}

It follows from the momentum equation \eqref{m14} and the fact 
$u^d$, $g$ are an odd functions with respect to $x_d$ that 
\begin{align} 
\intTd{ \vr u^i (t, \cdot)  } &= \intTd{ \vr_0 u^i_0 },\ i \leq d - 1,\br 
\intTd{ \vr u^d (t, \cdot)  } &=  0 \ \mbox{for any}\ t \in [0,T].
\label{p4}
\end{align}	

Next, we report the following version of Poincar\' e inequality, 
see \cite[Lemma 3.1]{FL3}.

\begin{Lemma} [\bf Poincar\' e inequality] \label{Lp1}
Let $r \geq 0$ be a given function on $\mathbb{T}^d$ such that 
\begin{equation} \label{p5}
0 < M \leq \intTd{ r },\ \intTd{ r^{\gamma} } \leq K,\ 
\gamma = \frac{6}{5} \ \mbox{if}\ d = 3, \ \gamma 
\ \mbox{arbitrary finite if}\ d = 2.
\end{equation}

Then there is a constant $C_P = C_P(M,K)$ depending on $M$ and $K$ but not on a specific shape of $r$ such that 
\begin{equation} \label{p6}
\left\| v - \left( \intTd{r} \right)^{-1} \intTd{ r v } \right\|_{L^2(\mathbb{T}^d)} \leq C_P (M,K) \Big\| \Grad v \Big\|_{L^2(\Td; R^d)} 
\end{equation}
for any $v \in W^{1,2}(\Td)$. 	
\end{Lemma}	

From now on, and without loss of generality, we shall always assume that the initial momentum has zero mean, 
\begin{equation} \label{p7}
\intTd{ \vr_0 \vu_0 } = 0.
\end{equation}	
Consequently, it follows from \eqref{p4}, \eqref{p5}
\begin{equation} \label{p8}
\| \vu (t, \cdot) \|_{W^{1,2}(\Td; \R^d)}^2 \leq \nu 2 C_P(M,K) 
\intTd{ \mathbb{S}(\Ds \vu): \Grad \vu (t, \cdot) }
\end{equation}
for any weak solution of the fluid--structure interaction problem as 
long as the associated density $\vr_0$, 
and consequently $\vr(t, \cdot)$ for any $t \in [0,T]$, satisfy \eqref{p5}.

\begin{Remark} \label{Rp1}
In view of the Galilean invariance of the problem, the restriction 
\eqref{p7} is not essential. Indeed as the initial momentum mean 
\[
\vc{M} = \intTd{ \vr_0 \vu_0 } 
\]	
is a conserved quantity, the desired solution with $\vc{M} \ne 0$ can be recovered by a simple transformation 
\[
\vr \approx \vr (t, x - t \vc{M}) ,\ 
\vr \vu = \vc{M} + \vr \vu (t, x - t \vc{M}).
\]

\end{Remark}	 

\subsection{Energy inequality revisited} 

To begin, let us check the validity of \eqref{p5}, \eqref{p6} 
in the context of Theorems \ref{Tm1}, \ref{Tm3}. As we shall show in Section \ref{w}, hypothesis \eqref{p5} is satisfied 
with $\gamma = \frac{3}{2}>\frac{6}{5}$ if $d=3$, $\gamma > 1$ if $d=2$. Indeed,
as the density satisfies the transport equation \eqref{m6}, both 
$\| \vr \|_{L^1(\Td)}$ and $\| \vr \|_{L^{\gamma}(\Td)}$ are constants of motion, and it is therefore enough to verify \eqref{p5} only for the initial data.
Consequently, under the hypothesis \eqref{p7}, we can rewrite the energy inequality \eqref{m9} in the form 
\begin{align} \label{p10}
	- \int_0^T \partial_t \psi &\intTd{  \frac{1}{2} \vr |\vu|^2 } + 
	\Lambda \int_0^T \psi \intTd{ \Big( | \Grad \vu |^2 + |\vu|^2 \Big) } \dt \br &\leq \psi(0) 
	\intTd{ \vr_0 |\vu_0|^2 } + \int_0^T \psi \intTd{ \vr \vc{g} \cdot \vu } \dt,\  \ \mbox{for some}\ \Lambda > 0,  
\end{align}	
for any $\psi \in C^1_c([0,T)$, $\psi \geq 0$.

\section{Clusters and cluster distance}
\label{C}

We recall that our physical domain is the flat torus 
$\mathbb{T}^d$. To simplify notation we set $H= L$. 
The distance on $\mathbb{T}^d$ is therefore given by 
\[
{\rm dist}_{\mathbb{T}^3}(x,y) = \left( \sum_{i=1}^d 
\left[ \min \Big\{2L -  |x_i - y_i|, |x_i - y_i|  \Big\} \right]^2 \right)^{\frac{1}{2}}
\]
If no confusion arises, we will use the standard Euclidean notation
${\rm dist}_{\mathbb{T}^d}(x,y) = |x - y|$.

\subsection{Cluster decomposition of a compact set}

Let $K \subset X$ be a compact subset of a complete metric space $X$ with metric $d$. We denote 
\[
B(r, h) = \left\{ x \in X \ \Big| \ d(x,h) < r \right\}, 
\]
and 
\[
\mathcal{U}[K, \delta] = \cup_{x \in K} B(\delta, x)
\]
and open $\delta-$neighbourhood of ${K}$.

\begin{Definition}[\bf Cluster decomposition of compact] \label{DC1}
Let 
\[
\mathcal{U}[K, \delta] = \cup_{\lambda \in \mathcal{K}} \mathcal{O}_{\lambda}(\delta)
\]
be a decomposition of the set $\mathcal{U}[K, \delta]$ into disjoint connected components. The family 
\[
\mathcal{C}_\lambda (\delta) = K \cap \mathcal{O}_{\lambda}(\delta) ,\ \lambda \in \mathcal{K} 
\]
is called \emph{cluster decomposition} of $K$ of size $\delta$.
	
\end{Definition}	

\subsection{Cluster distance}

It follows that
\begin{align} \label{C1}
x_1 , x_2 \in \mathcal{O}_\lambda (\delta) \ 
\Leftrightarrow \ &\mbox{there is a continuous curve}\ 
\gamma: [0,1] \to X \ \mbox{such that} \br 
& \gamma(0) = x_1,\ \gamma(1) = x_2,\ {\rm dist}[\gamma(t); 
K] < \delta \ \mbox{for all}\ t \in [0,1].
\end{align}	
This motivates the following definition of cluster distance. 

\begin{Definition}[\bf Cluster distance] \label{DC2}
For $x_1, x_2 \in X$, the \emph{cluster distance} with respect 
to a compact $K$ is defined as  
\begin{align} \label{C1a}
d_{K}[x_1, x_2] = \inf &\left\{ \delta > 0 \ \Big| 
\ \mbox{there is a continuous curve}\ \gamma: [0,1] \to X, \right. \br &\quad \gamma(0) = x_1,\ \gamma(1) = x_2,\ {\rm dist}[\gamma(t); 
K] < \delta \ \mbox{for all}\ t \in [0,1] \Big\}. 
\end{align}	
\end{Definition}

Note  that 
\[
d_{K}[x_1, x_2]	\geq 
\max \Big\{ {\rm dist}[x_1; K], {\rm dist} [x_2; K] \Big\}.
\]

The following properties are easy to check:
\begin{enumerate}
\item {\bf Pseudo-metric property.}
\begin{equation} \label{C2}
d_{K}[x_1, x_2] \geq 0,\ 	
d_{K}[x_1, x_2] = 0 \ \Leftrightarrow \ 
x_1, x_2 \ \mbox{belong to the same connected component of}\ K. 
\end{equation} 
\item {\bf Symmetry.}	
\begin{equation} \label{C3}
d_{K}[x_1, x_2] = d_{K}[x_2, x_1]
\end{equation}
\item {\bf Triangle inequality.} 
\begin{equation} \label{C4}
d_{K}[x, z] \leq d_{K}[x, y] + 
d_{K}[y, z].
\end{equation}	
\item {\bf Lipschitz continuity.}
\begin{equation} \label{C5}
\Big| 	d_{K}[x, z] - d_{K}[y, z] \Big| 
\leq |x - y|.
\end{equation}	

\end{enumerate}	

 Next, we examine continuity of the cluster distance with respect to the compact set $K$. We consider the Hausdorff metric on the 
set of all compact subsets of $X$, 
\[
d_{\rm Hausdorff}[K_1, K_2] = 
\max \left\{ \sup_{y \in K_1} {\rm dist}[y, K_2]; 
\sup_{z \in K_2} {\rm dist}[z, K_1] \right\}. 
\]
From this, we deduce
\[
d_{\rm Hausdorff}[K_1, K_2] \leq \ep \ \Leftrightarrow\ 
K_1 \subset \mathcal{U}[K_2, \ep] \ \mbox{and}\ 
K_2 \subset \mathcal{U}[K_1, \ep].
\]
Consequently, 
\[
d_{K_1}[x_1,x_2] \leq d_{K_2}[x_1,x_2] + d_{\rm Hausdorff}[K_1,K_2],
\]
and, similarly, 
\[
d_{K_2}[x_1,x_2] \leq d_{K_1}[x_1,x_2] + d_{\rm Hausdorff}[K_1,K_2];
\] 
whence
\begin{equation} \label{C6}
\Big| d_{K_1} [x,y] - d_{K_2} [x,y] \Big| 
\leq d_{\rm Hausdorff}[K_1, K_2].
\end{equation}	

Let us summarize the above observation concerning the dependence of $d_K[x,y]$ on the set $K$.

\begin{Lemma}[\bf Dependence on the set $K$] \label{LC1a}
\bigskip	
\begin{itemize}
	\item Let $\vc{h}: K \to X$ be continuous. Then the mapping
\[
\vc{h} \in C(K; X) \mapsto d_{K + \vc{h}}[x,y] 
\]
is Lipschitz continuous with the Lipschitz constant $1$, specifically, 
\begin{equation} \label{C66}
\left| d_{K + \vc{h}_1}[x,y] - d_{K + \vc{h}_2}[x,y] \right| \leq \max_{z \in K} {\rm dist}_X [ \vc{h}_1(z), \vc{h}_2(z)]. 
\end{equation}

\item Suppose $X$ is a Banach space and $\vc{h}(t,z) : [0,T] \times K \to X$, $\vc{h}(\cdot, z) \in W^{1,\alpha}(0,T; X)$, $0 < \alpha \leq 1$,  meaning
\[
\vc{h}(t_1, z) - \vc{h}(t_2, z) = \int_{t_1}^{t_2} \vc{h}'(t,z) \ \dt,\ \| \vc{h}'( \cdot ,z) \|_X \in L^\alpha(0,T) \ \mbox{for any}\ z \in K.
\]
Then
\begin{equation} \label{C67}
\left| d_{K + \vc{h}(t_2, \cdot)}[x,y] - d_{K + \vc{h}(t_1, \cdot)}[x,y] \right| 
\leq \sup_{z \in K} \int_{t_1}^{t_2} \| \vc{h}' (t,z) \|_{X} \dt.
\end{equation}
In particular, if $\sup_{z \in K} \| \vc{h}'( \cdot ,z) \|_X \in L^\alpha(0,T) $, then
\begin{equation} \label{C68}	
t \mapsto d_{K + \vc{h}(t, \cdot)}[x,y] \in W^{1,\alpha}(0,T),\ \mbox{and}\ 
\left| \frac{\D }{\dt}  d_{K + \vc{h}(t, \cdot)} \right|  \leq  \sup_{z \in K} \| \vc{h}' (t,z) \|_{X}\ \mbox{for a.a.}\ t \in (0,T).
\end{equation}
\end{itemize}		
	
\end{Lemma}	

\begin{proof}

Estimate \eqref{C66} follows from \eqref{C6} as soon as we observe that 
\[
d_{\rm Hausdorff}[K + \vc{h}_1, K + \vc{h}_2] \leq \max_{z \in K} {\rm dist}_X [ \vc{h}_1(z), \vc{h}_2(z)].
\]
Inequality \eqref{C67} is a direct consequence of \eqref{C66}. Finally, if $\sup_{z \in K} \| \vc{h}'( \cdot ,z) \|_X \in L^\alpha(0,T)$, then 
the function $t \mapsto d_{K + \vc{h}(t, \cdot)}$ is absolutely continuous and \eqref{C68} follows from \eqref{C67}.

\end{proof}

Finally, we need the following estimate of  the standard distance in terms of the cluster distance. 

\begin{Lemma}{\bf Bound on cluster distance} \label{LC1}

	Let $K=\cup_{n=1}^N K_n$, where 
$(K_n)_{n=1}^N$ are mutually disjoint compact sets. Then, for all $x,y\in X$, we have 
\begin{equation} \label{C8}
|x - y| \leq   \sum\limits_{n=1}^{N} \left(  {\rm diam} (K_n) + 
2 d_K[ x,y ] \right).
\end{equation}	
\end{Lemma}
\begin{proof}
Let us start with a simple observation: let $\mathcal{O}=\cup_{m=1}^{M} B_m$, where $\mathcal{O}$ is a open connected set and $B_m$ are open. Then, we have 
\begin{equation}\label{prag1}
    {\rm diam} (\mathcal{O}) \leq \sum\limits_{m=1}^M {\rm diam} (B_m).
\end{equation}
Indeed, as \eqref{prag1} clearly holds for $M=1$, assume the validity of \eqref{prag1} for $M$. Let us write 
\begin{equation*}
    \mathcal{O}=\cup_{m=1}^{M+1} B_m= \cup_{m=1}^{M} B_m \cup B_{M+1}, 
\end{equation*}
where $\cup_{m=1}^{M} B_m \cap B_{M+1} \neq \emptyset$. Thus, by induction we can conclude \eqref{prag1}. 

Now, for any $d> d_K[ x,y ]$ , where $x,y$ belong to a connected component $\mathcal{O}(d)$ of  $\mathcal{U}[K, d] = \cup_{n=1}^N \mathcal{U}[K_n, d]$. Consequently, $\mathcal{O}(d) = \cup_{m=1}^M \mathcal{U}[K_{n_m}, d]$, for some $M\leq N$. It follows from \eqref{prag1} that 
\begin{multline*}
    |x-y|\leq {\rm diam} (\mathcal{O}(d)) \leq \sum\limits_{m=1}^M {\rm diam} (\mathcal{U}[K_{n_m}, d]) \leq \sum\limits_{m=1}^M {\rm diam} (K_{n_m}) + 2d \leq  \sum\limits_{n=1}^N {\rm diam} (K_{n}) + 2d.
\end{multline*}
As $d > d_K[x,y]$ was arbitrary, the desired \eqref{C8} follows.
\end{proof}

\subsection{Clusters in the fluid--interaction problem}
\label{F}

To apply the abstract theory developed in the preceding section, we consider the metric space $X = \mathbb{T}^d$, where the compact set 
will be represented by $N$ points 
\begin{equation} \label{F1}
K = \{ X_1, \dots, X_N \},  
\ | X_i - X_j |
\geq 2 r \ \mbox{for}\ i \ne j. 
\end{equation}
It follows from \eqref{C2} that 
\begin{equation} \label{F2}
d_K [x_1, x_2] = 0 \ \Leftrightarrow\ 
x_1 = x_2 = X_n \ \mbox{for some}\ n = 1, \dots, N.
\end{equation}
Moreover, as $d_K[ X_i, X_i] = 0$, we deduce from 
\eqref{C5} with $z=y$ that 
\begin{equation} \label{F3a}
	d_K[  X_i, X_j ] \leq |X_i - X_j| ,\ i,j = 1, \dots N.
\end{equation}	
Finally, it follows from Lemma \ref{LC1} that
\begin{equation} \label{F3b} 
|X_i - X_j| \leq 2 N d_K[ X_i, X_j] ,\ i,j = 1, \dots N.
\end{equation}

\subsubsection{Cluster projection}


For any given point $X_i\in K = \{ X_1, \dots, X_N \}, i=1,\dots, N$, we define its \emph{cluster projection} with respect to $K = \{ X_1, \dots, X_N \}$ as 
\begin{equation} \label{F33}
X_{B,\delta} (X_i) = \left[ \sum_{n=1}^N \chi_\delta \Big( d_K [X_n, X_i] \Big) \right]^{-1}
\sum_{n=1}^N \left[ \chi_\delta \Big( d_K [X_n, X_i] \Big) X_n \right]
\end{equation}	
where $\chi_\delta$ is a cut-off function,
\begin{equation} \label{F4}
\chi_\delta(Z) = \chi \left( \frac{Z}{\delta} \right), \ \mbox{where}\ \chi \in C^\infty[0, \infty), \ \chi' \leq 0,\  
\chi(Z) = \left\{ \begin{array}{l} 1 \ \mbox{for}\ 
Z \in [0, 1] , \\ \exp(-Z) \ \mbox{for}\  Z \geq 2.
\end{array} \right\}	
\end{equation}	 
Note carefully that $d_K (X_i, X_i) = 0$ so the denominator in 
\eqref{F33} is always positive.

The function $\chi_\delta$ can be seen as a smooth approximation of the characteristic function $\mathds{1}_{[0, \delta]}$, for which 
the expression \eqref{F33} can be interpreted as the barycentre 
of the cluster of size $\delta$. As we shall see below, the exponential shape of $\chi_\delta$ is crucial to optimize certain estimates.

Our next objective is to estimate the distance 
\[
\Big| X_i - X_{B, \delta} (X_i) \Big| \ \mbox{for}\ 
i = 1, \dots, N. 
\]
A straightforward manipulation yields 
\begin{align} 
\Big| X_i - X_{B, \delta} (X_i) \Big|  
&\leq \left( \sum_{n=1}^N \chi_\delta \Big( d_K [X_n, X_i] \Big) \right)^{-1} \sum_{n=1}^N \chi_\delta \Big( d_K [X_n, X_i] \Big) 
| X_i - X_n | \br 
&=  \left( \sum_{n=1}^N \chi_\delta \Big( d_K [X_n, X_i] \Big) \right)^{-1} \sum_{\stackrel{n=1}{ d_K[X_n, X_i] \leq d}}^N \chi_\delta \Big( d_K [X_n, X_i] \Big) 
| X_i - X_n | \br 
&+  \left( \sum_{n=1}^N \chi_\delta \Big( d_K [X_n, X_i] \Big) \right)^{-1} \sum_{\stackrel{n=1}{ d_K[X_n, X_i] > d}}^N \chi_\delta \Big( d_K [X_n, X_i] \Big) 
| X_i - X_n | , 
\label{F5}
\end{align}
where $d = \delta |\log(\delta)| \geq 2 \delta$, $0 < \delta << 1$. On the one hand, using \eqref{F3b} we get 
\begin{align} 
 \left( \sum_{n=1}^N \chi_\delta \Big( d_K [X_n, X_i] \Big) \right)^{-1} \sum_{\stackrel{n=1}{ d_K[X_n, X_i] \leq d}}^N \chi_\delta \Big( d_K [X_n, X_i] \Big) 
| X_i - X_n | \leq 2 Nd. 
\label{F5a}
\end{align}	
On the other hand, 
\begin{align}
 \left( \sum_{n=1}^N \chi_\delta \Big( d_K [X_n, X_i] \Big) \right)^{-1} \sum_{\stackrel{n=1}{ d_K[X_n, X_i] > d}}^N \chi_\delta \Big( d_K [X_n, X_i] \Big) 
| X_i - X_n | \leq 2L N \exp \left( - \frac{d}{\delta} \right). 
\label{F5b}
\end{align}
Combining \eqref{F5}--\eqref{F5b} we conclude 
\begin{equation} \label{F5c}
| X_i - X_{B, \delta} (X_i) | \leq C(L) N \delta |\log(\delta)|
\ \mbox{for all}\ 0 < \delta < \frac{1}{2},\ i = 1 , \dots, N. 
\end{equation}

\subsubsection{Dependence of the cluster projection on parameters}

We consider the cluster projection 
\[
X_{B, \delta}(X_i) 
\ \mbox{as a function of the parameters}\ X_1, \dots, X_n,\ 
Z_n = d_K [X_n, X_i]. 
\] 

It is a routine matter to compute
\begin{equation} \label{F11}
\frac{ 	D	X_{B, \delta} (X_i) }{ D X_i} =
\left( 1 + \sum_{n \ne i} \chi_\delta \Big( d_K [X_n, X_i] \Big) \right)^{-1}, 
\end{equation}
\begin{equation} \label{F12}
	\frac{ 	D X_{B, \delta} (X_i)	}{ D X_n} =
	\left( 1 + \sum_{j \ne i}\chi_\delta \Big( d_K [X_j, X_i] \Big) \right)^{-1}  \chi_\delta \Big( d_K[ X_n, X_i] \Big) \ \mbox{for} 
	\ n\ne i,
\end{equation}
and 
\begin{align}
&\frac{D X_{B, \delta} (X_i)  }{D Z_{n}} = 
\left( \sum_{j=1}^N \chi_\delta \Big( d_K [X_j, X_i] \Big) \right)^{-2}
\left[ \sum_{j = 1}^N \chi_\delta \Big( d_K [X_j, X_i] \Big) (X_n - X_j) \right] \chi'_\delta( d_K [X_n, X_i] ).
\label{F13}
\end{align}	
Moreover, the norm of
the right--hand side in \eqref{F13} can estimated as
\begin{align}
	&\left| \frac{D X_{B, \delta} (X_i)  }{D Z_{n}} \right| = \left|
	\left(\sum_{j=1}^N \chi_\delta \Big( d_K [X_j, X_i] \Big) \right)^{-2}
	\left[ \sum_{j = 1}^N \chi_\delta \Big( d_K [X_j, X_i] \Big) (X_n - X_j) \right] \chi'_\delta( d_K [X_n, X_i] ) \right|  \br 
&{ \leq \frac{1}{\delta} \sup_{n} \left| \frac{\chi'_\delta ( d_K[X_n, X_i])} 
{ \chi_\delta ( d_K[X_n, X_i]) } \right|\chi_\delta( d_K [X_n, X_i] )
\left(\sum_{j=1}^N \chi_\delta \Big( d_K [X_j, X_i] \Big) \right)^{-1}
\Big| X_n - X_{B, \delta}[X_n] \Big| }\br 
& {\aleq N \log(\delta) \ \chi_\delta( d_K [X_n, X_i] )
\left(\sum_{j=1}^N \chi_\delta \Big( d_K [X_j, X_i] \Big) \right)^{-1}.}
	\label{F23}
\end{align}	

\subsection{Time dependence of cluster projections}

Going back to the original setting of the fluid--structure interaction problem, we consider the cluster projections related to the centres of the rigid balls,  
\[
X_n  = \vc{h}_n(t),\ n=1, \dots, N, 
\] 
with the associated functions 
\[
t \in [0,T] \mapsto X_{B, \delta}[\vc{h}_n(t)] \in \Td.
\]
 Combining the estimates on Lipschitz continuity  $d_K(x,y)$ with respect to $(x,y)$ stated in \eqref{C5},
together with the estimates obtained in Lemma \ref{LC1a}
we obtain the following result.

\begin{Lemma} [\bf Time differentiability of cluster projections] \label{LF1}

Let $\vc{h}_i$, $i = 1, \dots, N$ be continuous functions of the time $t \in [0,T]$ belonging to class $\vc{h}_i \in W^{1, \alpha}([0,T]; \Td)$, $\alpha > 1$.

Then $X_{B, \delta}(\vc{h}_i)$ are differentiable, specifically 
$X_{B, \delta}(\vc{h}_i) \in W^{1, \alpha}([0,T]; \R^d)$, 
and the following estimate
\begin{equation} \label{F24} 
\left| \frac{\D }{\dt} X_{B, \delta}(\vc{h}_i (t)) \right| \leq 
C N |\log(\delta)| \max_{n =1, \dots N} \left| \frac{\D }{\dt} 
\vc{h}_n (t) \right| \ \mbox{for a.a}\ t \in (0,T), 
\end{equation}
holds for any $i = 1, \dots, N$.
	
\end{Lemma}	

\section{Approximate test functions}
\label{A}

In this section, we construct suitable approximations of 
smooth functions eligible for the test functions in the momentum 
balance \eqref{m7}. They represent a time dependent analogue 
of the \emph{restriction operators} used in static homogenization. 

We start by a crucial property of the cluster projections 
$X_{B, \delta}(X_i)$, namely $X_{B, \delta}(X_i)$ coincides with 
$X_{B, \delta} (X_j)$ as soon as the points $X_i$, $X_j$ are close enough. This is a consequence of a series of elementary observations:
\begin{enumerate}
\item 
\begin{equation} \label{A1}
	d_K(X_i, X_j) \leq |X_i - X_j| ;
\end{equation} 
\item
\begin{equation} \label{A2}
|X_i - X_j| < \delta \ \Rightarrow \ d_K(X_i, X_j) < \delta 
\Rightarrow \ \chi_\delta ( d_K(X_i, X_j) ) = 
\chi_\delta ( d_K(X_j, X_i) ) = 1; 
\end{equation}	
\item 
\begin{equation} \label{A3}
|X_i -  X_j| < d_K(X_i, X_m) \ \Rightarrow \ 
d_K( X_j , X_m) = d_K(X_i, X_m)  
\end{equation}	
Indeed obviously 
\[
d_K (X_j, X_m) \leq d_K (X_i, X_m) \ \mbox{and, symmetrically}\ 
d_K (X_i, X_m) \leq d_k (X_j, X_m).
\]	

\end{enumerate}	

Summing up the above observation, we deduce the following crucial result. 

\begin{Lemma} \label{LA1}
Let $|X_i - X_j| < \delta$. 

Then	
\[
X_{B, \delta}(X_i) = X_{B, \delta}(X_j).
\]

	\end{Lemma}	

\begin{proof}
First, it follows from \eqref{A2} that 
\[
\chi_\delta (d_K (X_i, X_j)) = \chi_\delta (d_K (X_j, X_i)) = 1.
\]

Next, by virtue of \eqref{A3}, 
\[
\chi_\delta (d_K (X_i, X_m)) = 
\chi_\delta (d_K (X_j, X_m)) \ \mbox{whenever}\ 
d_K (X_i, X_m) \geq \delta.
\]

Finally, if $d_K(X_i, X_m) < \delta$, we use 
\eqref{A2} to deduce 
\[
\chi_\delta (d_K(X_i, X_m)) = 1, 
\]
and, similarly, $\chi_\delta (d_K(X_j, X_m)) = 1$ whenever 
$d_K(X_j, X_m) < \delta$.

Summarizing, we get 
\[
\chi_\delta (X_m, X_i) = \chi_\delta (X_m, X_j)
\ \mbox{for any}\ m = 1, \dots , N
\]
which yields the desired result.

\end{proof}	

\subsection{Construction of approximate test functions - principal ideas} 
\label{cat}

Given a (smooth) function $\bfphi \in C^\infty(\Td; \R^d)$, $\Div \bfphi = 0$ and the centres of balls $\vc{h}_n$ 
of radius $r > 0$, our goal is to construct 
an admissible (smooth) test function $\widetilde{\bfphi}$ satisfying 
\begin{equation} \label{A4}
\widetilde{\bfphi} \approx \bfphi,\ \Div \widetilde{\bfphi} = 0,\ 
\Ds \widetilde{\bfphi} = 0 \ \mbox{on an open neighbourhood of}\ 
\cup_{n=1}^N \Ov{B(r, \vc{h}_n)}.
\end{equation}

We sketch the main idea of the construction based on the concept of cluster, in particular the cluster projections $X_{B, \delta}(\vc{h}_i)$.
\begin{enumerate}
\item We identify a ``partition of unity'' - a system of functions $R_n$, $n = 1,\dots, N$ such that 
\begin{align}
0 &\leq R_n \leq 1,\ 
\sum_{n=1}^N R_n(x) = 1 \ \mbox{if}\ x 
\in \cup_{n=1}^N \Ov{ B( \beta r, \vc{h}_n) },\ \mbox{for some}\ \beta > 1 \br
R_n(x)& = 0 \ \mbox{whenever}\ |x - \vc{h}_n| > 4 r,\ 
n = 1, \dots, N. 
\label{A5}
\end{align}
The first naive approximation of $\bfphi$ is 
\[
\widetilde{\bfphi} = \bfphi (1 - \sum_{n=1}^N R_n) .
\]
It is easy to check that 
\begin{align}
\widetilde{\bfphi} (x) &= 0 \ \mbox{whenever}\ 
|x - \vc{h}_n| <  \beta r \ \mbox{for some}\ n,\br 
\widetilde{\bfphi} (x) &= \bfphi(x) \ \mbox{whenever}\ 
|x - \vc{h}_n| > 4r \ \mbox{for all}\ n = 1, \dots, N.
\nonumber
\end{align}
Obviously, the approximation vanishes on all rigid balls, however, 
the error 
\begin{equation} \label{A6}
\widetilde{\bfphi} - \bfphi = \bfphi \sum_{n=1}^N R_n 
\end{equation}
is not optimal. 

\item To improve \eqref{A6}, we exploit the fact that condition \eqref{A4} permits approximations by rigid motions which, when restricted to the balls, can in particular be chosen to be constant. Consequently, we may consider the approximation in the form 
\begin{equation} \label{A7}
\widetilde{\bfphi} = \bfphi \left( 1 - \sum_{n=1}^N R_n \right) + 
\sum_{n=1}^N \bfphi (\vc{h}_n) R_n. 
\end{equation}
The error 
\[
\widetilde{\bfphi} - \bfphi = \sum_{n=1}^N \left(\bfphi - \bfphi(\vc{h}_n) \right) R_n 
\]
is definitely better than \eqref{A6} as long as $\bfphi$ is at least Lipschitz as
\[
\Big| \bfphi(x) - \bfphi (\vc{h}_n) \Big| \leq 4 \| \Grad \bfphi \|_{L^\infty} r
\ \mbox{for any}\ x \in B(4r, \vc{h}_n) \ \Rightarrow 
| \widetilde{\bfphi} - \bfphi | \leq 4 \| \Grad \bfphi \|_{L^\infty} 
r \ \mbox{in}\ \Td.
\]

The main problem of the construction \eqref{A7} is that, in general, $\widetilde{\bfphi}$ \emph{is not} rigid on a neighbourhood of $\cup_{n=1}^N \Ov{B(r, \vc{h}_n)}$ as long as two or more centres are close, say $|\vc{h}_i - \vc{h}_j| < 3r$, and  
$\bfphi (\vc{h}_i) \ne \bfphi (\vc{h}_j)$. 
The remedy of the situation may be seen as the principal new idea of the present paper. 
In view of Lemma \ref{LA1}, we replace the centres $\vc{h}_i, \vc{h}_j$ by their cluster projections 
$X_{B, 4 r}[ \vc{h}_i]$, $X_{B,4  r}[\vc{h}_j]$. The approximation now reads 
\begin{equation} \label{A8}
	\widetilde{\bfphi} = \bfphi \left( 1 - \sum_{n=1}^N R_n \right) + 
	\sum_{n=1}^N \bfphi \Big(X_{B, 4  r}(\vc{h}_n) \Big) R_n. 
\end{equation}  
Thanks to Lemma \ref{LA1}, 
\[
X_{B, 4 r}(\vc{h}_i) = 
X_{B, 4 r}(\vc{h}_j) \ \mbox{whenever}\ 
|\vc{h}_i - \vc{h}_j| < 4 r,  
\]
in particular $\Ds \widetilde \bfphi = 0$ on 
$B(\beta r, \vc{h}_n)$ for any $n = 1,\dots, N$ for a certain 
$\beta > 1$ depending on the system $(R_n)_{n=1}^N$.

\item The construction \eqref{A8} yields a test function with all 
properties required in \eqref{A4} except solenoidality, meaning 
$\Div \widetilde{\bfphi} \ne 0$ in $\Td$. As a matter of fact, 
$\widetilde{\bfphi}$ \emph{is} solenoidal in the $\beta r$ neighbourhood of rigid balls $\cup_{n=1}^N \Ov{ B(\beta r, \vc{h}_n) }$ and also in the set $\Td \setminus \cup_{n=1}^N \Ov{ B( 4r, \vc{h}_n) }$.  

There are two ways how to solve this problem. The first possibility is to apply a suitable inverse of $\Div-$operator, for instance the so-called Bogovskii operator (see \cite{BOG}),  
on $\Div \widetilde{\bfphi}$ in the set 
$\cup_{n=1}^N { B( 4r, \vc{h}_n) } \setminus \Ov{ B(\beta r, \vc{h}_n) }$. The second option, adapted in the present paper, is 
``lifting'' the problem to the potential $\bfpsi$, $\Curl \bfpsi = \bfphi$, replacing $\bfpsi$ by a suitable $\widetilde{\bfpsi}$, 
and setting $\widetilde{\bfphi} = \Curl \widetilde{\bfpsi}$.
\end{enumerate}	

\subsection{Decomposition of unity on rigid balls}

Our first task is to construct the functions $R_n$ enjoying the 
property \eqref{A5}. The building block is the standard cut-off 
function $\Phi$, 
\begin{align}
\Phi &\in C^\infty_c[0, \infty),\ 
\Phi (y) = 1 \ \mbox{for}\ 0 \leq y \leq \beta ,\ 
\Phi (y) = 0 \ \mbox{for}\  y \geq 2,\ \Phi' (y) \leq 0, \br  
\mbox{where} \ 1 &< \beta < 2,\ 
\Phi_r (y) = \Phi \left( \frac{y}{r} \right).  
\nonumber
\end{align} 
We set 
\begin{align}
R_n (x) = \Phi_r (| x - \vc{h}_n|) \Pi_{n < j \leq N} \Big( 1 - \Phi_r (| x - \vc{h}_j|) \Big) ,\ 1 \leq n < N,\ 
{R_N(x) = \Phi_r (| x - \vc{h}_N|)}.	 
\label{A9} 
\end{align}	
Obviously, {
\begin{equation*}
|x - \vc{h}_n| > 2r \ \Rightarrow \ R_n (x) = 0. 
\end{equation*}
}
Finally, we introduce the partial sums 
\[
S_M = \sum_{n=1}^M R_n. 
\]
It is easy to check that 
\begin{equation} \label{A11}
S_N = S_{N-1} \Big(1 - \Phi_r(| x - \vc{h}_N|) \Big) + 
\Phi_r(| x - \vc{h}_N|).
\end{equation}	
We claim the following:
\begin{equation} \label{A12}
\Phi_r(| x - \vc{h}_m|)  = 1 
\ \mbox{for some}\ m \ \Rightarrow \ 
S_N (x) = \sum_{n=1}^N R_n(x) = 1. 
\end{equation}
In view of \eqref{A11}, the claim \eqref{A12} is easy to prove by induction. As the statement is obvious for $N=1$, we suppose 
\eqref{A12} is valid for $N = 1$. But then either $\Phi_r(| x - \vc{h}_N|) = 1$ or, by induction hypothesis, $S_{N-1}(x) = 1$; whence \eqref{A11} yields the claim.

\subsection{Construction of approximate test functions}

Let $\bfphi = \bfphi (t,x)$ be a (sufficiently) smooth function, $\Div \bfphi (t, \cdot) = 0$, $\intTd{\bfphi (t, \cdot)} = 0$
for any $t \in [0,T]$. To simplify 
presentation, we restrict ourselves to the case $d = 3$. 
However, the same construction applies if $d = 2$ with obvious modifications.
 
To guarantee solenoidality, we ``lift'' the approximate process to the level 
of vector potential $\bfPsi$ obtained in the standard way: 
\[
\bfphi = \Curl \bfPsi, \ \mbox{meaning}\ 
\Del \bfPsi = - \Curl \bfphi ,\ \Div \bfPsi = 0,\ \intTdd{ \bfPsi } = 0. 
\]  

Following the ideas delineated in Section \ref{cat}, we set 
\begin{align} 
\bfPsi_{\rm app}(t,x) &= \bfPsi (t,x) \left[ 1 - \sum_{n=1}^N 
\Phi_r (| x - \vc{h}_n(t)|) \Pi_{n < j \leq N} \Big( 1 - \Phi_r (| x - \vc{h}_j(t)|) \Big)\right] \br 
&+\sum_{n=1}^N \left[
\bfPsi \Big(t, X_{B,4r}(\vc{h}_n(t))\Big)
+
\Grad \bfPsi \Big(t, X_{B,4r}(\vc{h}_n(t))\Big)
\cdot
\Big(x - X_{B,4r}(\vc{h}_n(t))\Big)
\right] \times \br
&\quad \times 
\Phi_r (| x - \vc{h}_n(t)|) \Pi_{n < j \leq N} \Big( 1 - \Phi_r (| x - \vc{h}_j(t)|) \Big).
\label{A13}
	\end{align}	
Finally, we set
\begin{equation} \label{A14}
\bfphi_{\rm app}(t,x) = \Curl \bfPsi_{\rm app} (t,x).
\end{equation}

The function $\bfphi_{\rm app}$ obviously inherits the regularity of 
$\bfphi$ and satisfies 
\[
\Div \bfphi_{\rm app}(t, \cdot) = 0,\ \intTdd{ \bfphi_{\rm app} (t, \cdot) } = 0 \ \mbox{for any}\ t \in [0,T].
\]
Moreover, by virtue of \eqref{A12},
\[ 
\sum_{n=1}^N 
\Phi_r (| x - \vc{h}_n(t)|) \Pi_{n < j \leq N} \Big( 1 - \Phi_r (| x - \vc{h}_j(t)|) \Big) = 1 \ \mbox{whenever}\ 
|x - \vc{h}_i| < \beta r 
\]
for $i = 1, \dots, N$. { As $1 < \beta < 2$, fixing the index $i$ we observe 
\[
|x - \vc{h}_i| < \beta r \ \mbox{and}\ 
\Phi_{r}(| x - \vc{h}_n |) \ne 0 \ \mbox{for some}\ 
n \ne i \ \Rightarrow \ | \vc{h}_i - \vc{h}_n | \leq \beta r + 2 r < 4r. 
\] 
}
Consequently, by virtue of Lemma \ref{LA1}, 
{
\[
|x - \vc{h}_i| < \beta r \ \mbox{and}\ 
\Phi_{r}(| x - \vc{h}_n |) \ne 0 \ \mbox{for some}\ 
n \ne i \ \Rightarrow\ X_{B, 4r}(h_i) = X_{B, 4r}(h_n). 
\]
}
We conclude 
\begin{align}
&\sum_{n=1}^N \left[  \bfPsi \Big(t, X_{B, 4 r}(\vc{h}_n(t) ) \Big) + 
\Grad \bfPsi \Big( t, X_{B, 4r} (\vc{h}_n(t)) \Big) \cdot \Big(x - X_{B, 4r} (\vc{h}_n(t)) \Big)   \right] \times \br
&\quad \times 
\Phi_r (| x - \vc{h}_n(t)|) \Pi_{n < j \leq N} \Big( 1 - \Phi_r (| x - \vc{h}_j(t)|) \Big)  \br 
&\quad = \bfPsi \Big(t, X_{B, 4 r}(\vc{h}_i(t) ) \Big) + 
\Grad \bfPsi \Big( t, X_{B, 4r} (\vc{h}_i(t)) \Big) \cdot \Big(x - X_{B, 4r} (t, \vc{h}_i(t)) \Big) 
\nonumber
\end{align}
for any $x \in B(\beta r, \vc{h}_i)$ and any $i = 1,\dots, N$ yielding the desired conclusion 
\begin{equation} \label{A15}
\bfphi_{\rm app}(t,x) = (\Curl \bfPsi) (t, X_{B, 4r} (\vc{h}_i(t))) = 
\bfphi (t, X_{B, 4r} (\vc{h}_i(t))) 
\ \mbox{for} \ x \in B(\beta r, \vc{h}_i(t)),\ i = 1, \dots, N.
\end{equation}

\subsection{Rate of convergence of approximations}
\label{CA}

Our ultimate task in this section is to evaluate the distance $\bfphi_{\rm app} - 
\bfphi$ in terms of the parameters $r$, $N$. depend on a small parameter $\ep \to 0$. Note carefully that no information on the approximate densities is needed at this stage. 

\begin{Proposition} [\bf Order of approximation] \label{PCA1}
	
Let $\bfphi$ be a smooth function, specifically $\bfphi$, $\nabla^2_x \bfphi$, $\partial_t \bfphi$ continuous in $[0,T] \times \Td$, and
$\Div \bfphi (t, \cdot) = 0$, $\intTd{ \bfphi(t, \cdot) } = 0$ for any $t \in [0,T]$. Let 
 $\{\vc{h}_n \in W^{1,\alpha}(0,T; \R^d)$, $\alpha > 1\}_{n=1}^N$ be a family of functions such that
\begin{equation} \label{CA1}
r N \leq 1,\ |\vc{h}_i(t) - \vc{h}_j(t)| \geq 2r \ \mbox{for any}\ i \ne j
\ \mbox{and any}\ t \in [0,T] .
\end{equation}	

Then there holds:
\begin{itemize} 
\item	
{\bf Space derivative estimates.} 
\begin{equation} \label{CA3}
\sup_{t \in [0,T]} \| (\bfphi - \bfphi_{\rm app})(t, \cdot)  \|_{L^q(\Td; R^{d} )} \aleq 
\| D_x \bfphi \|_{L^\infty}  N^2 r |\log(r)|^2 N^{\frac{1}{q}} r^{\frac{d}{q}},
\end{equation}	
\begin{equation} \label{CA5}
\sup_{t \in [0,T]}	\| \Grad (\bfphi - \bfphi_{\rm app} )(t, \cdot) \|_{L^q(\Td; R^{d^2 } )} \aleq 
	\| D^2_x \bfphi \|_{L^\infty} N^2 |\log(r)|^2 N^{\frac{1}{q}} r^{\frac{d}{q}}
\end{equation}
for any finite $q \geq 1$.

\item {\bf Time derivative estimates.}
 \begin{equation} \label{CA7}
\| \partial_t ( \bfphi - \bfphi_{\rm app} )(t, \cdot) \|_{L^q(\Tdd; R^{d} )}  \aleq 
	\| D_t \bfphi \|_{L^\infty} N^2  |\log(r)|^2 N^{\frac{1}{q}} r^{\frac{d}{q} } \left( 
	1 + \max_{1 \leq n \leq N} \left| \frac{\D \vc{h}_n (t) }{\dt} \right| \right)
\end{equation}
for a.a. $t \in [0,T]$, and for any finite $q \geq 1$.

\item {\bf Estimates on rigid balls.}

\begin{equation} \label{CA88}
\sup_{t \in[0,T]}	\| (\bfphi - \bfphi_{\rm app}) (t, \cdot)  \|_{L^q(\cup_{n=1}^N B(r; \vc{h}_n(t) ); \R^d)} \aleq \| \bfphi \|_{L^\infty} r^2 |\log(r)| N
N^{\frac{1}{q}} r^{\frac{d}{q}} ,
\end{equation}
and
\begin{equation} \label{CA8}
	\| \partial_t  \bfphi_{\rm app} (t, \cdot)  \|_{L^q(\cup_{n=1}^N B(r; \vc{h}_n(t) ); \R^d)} \aleq
	C \| D_t \bfphi \|_{L^\infty} N^{\frac{1}{q}} r^{\frac{d}{q}} \left( 1 + N |\log(r)| \max_{1 \leq n \leq N} \left| \frac{\D \vc{h}_n (t) }{\dt} \right| 
	\right)
\end{equation}
for a.a. $t \in (0,T)$.
	
\end{itemize}	

\end{Proposition}

\begin{proof}

Revisiting formulae \eqref{A13}, \eqref{A14} we get 
\begin{align} 
&(\bfPsi - \bfPsi_{\rm app} )(t,x) \br &= 
\sum_{n=1}^N  R_n (t,x) \left[ \bfPsi(t,x) - \left(  \bfPsi \Big(t, X_{B, 4 r}(\vc{h}_n(t) ) \Big) + 
\Grad \bfPsi \Big( t, X_{B, 4r} (\vc{h}_n(t)) \Big) \cdot \Big(x - X_{B, 4r} (\vc{h}_n(t)) \Big)   \right) \right], 
\label{CAA1} 
\end{align}
where 
\[
R_n (t,x) = \Phi_{r}(|x - \vc{h}_n(t)|) \Phi_{n < j \leq N} 
\Big(1 - \Phi_r(|x - \vc{h}_j(t)| ) \Big).
\]
Obviously $R_n = 0$ as soon as $|x - \vc{h}_n| \geq 2r$ for all 
$\vc{h}_n$, $n=1,\dots, N$. Consequently, it is enough to show \emph{uniform} bounds for \eqref{CAA1}  
on each ball $B(2r; \vc{h}_n)$; the $L^q$ norm is then obtain by multiplying the result on 
\begin{equation} \label{CAALQ}
N^{\frac{1}{q}} r^{\frac{d}{q}} \approx \| \mathds{1}_{\cup_{n=1}^N B(2r, \vc{h}_n}) \|_{L^q(\Td)}.  
\end{equation}

\medskip

\noindent {\bf Step 1.} 
{
The first crucial observation is
\begin{equation} \label{CAA11}
x \in B(2r, \vc{h}_n) \ \mbox{and}\ \Phi_r (|x - \vc{h}_j(t)|) \ne 0 
\ \Rightarrow \ |\vc{h}_n - \vc{h_j}| \leq 4 r
\end{equation}

As the number of balls of radius $r$ that can be contained in a ball of radius $4r$ in 
$\R^d$ does not exceed 
\[
\frac{ c(d) (3 r)^d  }{ c(d) r^d} = 4^d,\  \ d =2,3, 
\]
we conclude that for each $i = 1, \dots, n$, there is a set $P(i) \subset \{ 1, \dots, N \}$, $\# P(n) \leq 4^d$ such that 
\begin{align} 
	&(\bfPsi - \bfPsi_{\rm app} )(t,x) \br &= 
	\sum_{n \in P(i)}  R_n (t,x) \left[ \bfPsi(t,x) - \left(  \bfPsi \Big(t, X_{B, 4 r}(\vc{h}_n(t) ) \Big) + 
	\Grad \bfPsi \Big( t, X_{B, 4r} (\vc{h}_n(t)) \Big) \cdot \Big(x - X_{B, 4r} (\vc{h}_n(t)) \Big)   \right) \right], \br 
	& \mbox{for any}\ x \in B(2r, \vc{h}_i),\ i = i, \dots, N.
	\label{CAA1A} 
\end{align}
Accordingly, the sum in \eqref{CAA1A} contains at most $4^d$ non--zero terms and it is enough to show the estimates for 
each
}
\begin{align}
&E_n (t,x) \br &= R_n (t,x) \left[ \bfPsi(t,x) - \left(  \bfPsi \Big(t, X_{B, 4 r}(\vc{h}_n(t) ) \Big) + 
\Grad \bfPsi \Big( t, X_{B, 4r} (\vc{h}_n(t)) \Big) \cdot \Big(x - X_{B, 4r} (\vc{h}_n(t)) \Big)   \right) \right]
\nonumber
\end{align}
$n \leq N$ fixed. 

\medskip 

\noindent{\bf Step 2.} We consider $\Curl E_n$ to obtain the approximation error for $\bfphi$:
\begin{align} 
&\Curl E_n(t,x) \br
& =  \Grad R_n (t,x) \times  	\left[ \bfPsi(t,x) - \left(  \bfPsi \Big(t, X_{B, 4 r}(\vc{h}_n(t) ) \Big) + 
\Grad \bfPsi \Big( t, X_{B, 4r} (\vc{h}_n(t)) \Big) \cdot \Big(x - X_{B, 4r} (\vc{h}_n(t)) \Big)   \right) \right] \br 
&\quad + R_n(t,x) \Big[ \bfphi(t,x) - \bfphi(t, X_{B, 4r}(\vc{h}_n(t)) \Big], 
\label{CA10}
\end{align}	 
where, by virtue of \eqref{F5c},  
\begin{align} 
&|\Grad R_n| \aleq \frac{1}{r} ,\ |R_n| \aleq 1, \br
&	\left| \bfPsi(t,x) - \left(  \bfPsi \Big(t, X_{B, 4 r}(\vc{h}_n(t) ) \Big) + 
\Grad \bfPsi \Big( t, X_{B, 4r} (\vc{h}_n(t)) \Big) \cdot \Big(x - X_{B, 4r} (\vc{h}_n(t)) \Big)   \right) \right| \br 
&\quad \aleq \|D^2_x \bfPsi \|_{L^\infty} \Big|x - \vc{h}_n(t) + \vc{h}_n(t) - 
X_{B, 4r}[\vc{h}_n(t)] \Big|^2 \aleq \|D_x \bfphi \|_{L^\infty} 
N^2 r^2 |\log(r)|^2, \br 
&\Big| \bfphi(t,x) - \bfphi(t, X_{B, 4r}(\vc{h}_n(t)) \Big|
\aleq \|D_x \bfphi \|_{L^\infty} \Big|x - \vc{h}_n(t) + \vc{h}_n(t) - 
X_{B, 4r}[\vc{h}_n(t)] \Big| \br 
&\quad \aleq \|D_x \bfphi \|_{L^\infty} N r  |\log(r)|. 
\label{CA11}	
\end{align}	
Using \eqref{CAALQ} we get \eqref{CA3}.

\medskip

\noindent {\bf Step 3.} To estimate the gradient of $\bfphi - 
\bfphi_{\rm app}$, it is enough to handle $\nabla(\Curl E_n)$. As $\Div(\Curl E_n)=0$, and 
$\Curl E_n$ is of zero mean, we may use Friedrichs--Gaffney inequality to 	 
control $\|\nabla(\Curl E_n)\|_{L^p}$ by only $ \|\Curl \Curl E_n\|_{L^p}$.
A direct computation yields
\begin{align} 
&\Curl \Curl E_n \br &= - \left[ \bfPsi(t,x) - \left(  \bfPsi \Big(t, X_{B, 4 r}(\vc{h}_n(t) ) \Big) + 
\Grad \bfPsi \Big( t, X_{B, 4r} (\vc{h}_n(t)) \Big) \cdot \Big(x - X_{B, 4r} (\vc{h}_n(t)) \Big)   \right) \right] \Del R_n(t,x)
\br 
&+ \left[ \bfPsi(t,x) - \left(  \bfPsi \Big(t, X_{B, 4 r}(\vc{h}_n(t) ) \Big) + 
\Grad \bfPsi \Big( t, X_{B, 4r} (\vc{h}_n(t)) \Big) \cdot \Big(x - X_{B, 4r} (\vc{h}_n(t)) \Big)   \right) \right]\cdot \Grad ( \Grad R_n (t,x) )  \br
&- \Grad R_n (t,x) \cdot \left[ \Grad \bfPsi(t,x) - 
\Grad \bfPsi \Big( t, X_{B, 4r} (\vc{h}_n(t)) \Big)  \right]\br  
& +	\Grad R_n(t,x) \times \Big[ \bfphi(t,x) - \bfphi(t, X_{B, 4r}(\vc{h}_n(t)) \Big]  + R_n(t,x) \Curl \bfphi(t,x). 
\label{CA14}
\end{align}	
Seeing that 
\[
|\Grad^2 R_n| \aleq \frac{1}{r^2} 
\]
we may repeat the arguments \eqref{CA11} to deduce \eqref{CA5}. 

\medskip \noindent {\bf Step 5.} 

Finally, we estimate the error for the time derivative. To this end, 
we write $R_n = R_n(t,x, X_n)$, $X_n = \vc{h}_n$ as a function of three variables. Differentiating \eqref{CA10} we get
\begin{align}
&\partial_t \Curl E_n  = 
\partial_t \Grad R_n (t,x, \vc{h}_n) \times  \br &\times	\left[ \bfPsi(t,x) - \left(  \bfPsi \Big(t, X_{B, 4 r}(\vc{h}_n(t) ) \Big) + 
\Grad \bfPsi \Big( t, X_{B, 4r} (t,\vc{h}_n(t)) \Big) \cdot \Big(x - X_{B, 4r} (\vc{h}_n(t)) \Big)   \right) \right] \br
&+ \Grad R_n (t,x, \vc{h}_n) \times \br 
&\times	\partial_t \left[ \bfPsi(t,x) - \left(  \bfPsi \Big(t,X_{B, 4 r}(\vc{h}_n(t) ) \Big) + 
\Grad \bfPsi \Big( t, X_{B, 4r} (\vc{h}_n(t)) \Big) \cdot \Big(x - X_{B, 4r} (\vc{h}_n(t)) \Big)   \right) \right]\br &+\partial_t R_n(t,x) \Big[ \bfphi(t,x) - \bfphi(t, X_{B, 4r}(\vc{h}_n(t))\Big] + R_n(t,x) \partial_t\Big[ \bfphi(t,x) - \bfphi(t, X_{B, 4r}(\vc{h}_n(t)) \Big]. 
\label{CA15}
\end{align}	 
Furthermore, 
\begin{equation} \label{CA16}
\left| \partial_t \Grad R_n (t,x, \vc{h}_n) \right| 
\aleq r^{-2} \max_{n=1, \dots, N} \left| \frac{\D \vc{h}_n(t)}{\dt} \right|  ,
\end{equation}	
and 
\begin{equation} \label{CA16a}
\left| \bfPsi(t,x) - \left(  \bfPsi \Big(t, X_{B, 4 r}(t, \vc{h}_n(t) ) \Big) + 
\Grad \bfPsi \Big( t, X_{B, 4r} (\vc{h}_n(t)) \Big) \cdot \Big(x - X_{B, 4r} (\vc{h}_n(t) )\Big)   \right) \right| \aleq N^2 r^2 |\log(r)|^2 .
\end{equation}	

Moreover, by virtue of \eqref{F24},  
\begin{align} 
&\left| \partial_t \left[ \bfPsi(t,x) - \left(  \bfPsi \Big(t, X_{B, 4 r}(\vc{h}_n(t) ) \Big) + 
\Grad \bfPsi \Big( t, X_{B, 4r} (\vc{h}_n(t)) \Big) \cdot \Big(x - X_{B, 4r} (\vc{h}_n(t)) \Big)   \right) \right]	\right| 
\br &\aleq C \| \partial_t \Grad \Psi \|_{L^\infty} 
(1 + N |\log(r) |) 
\label{CA17}
\end{align}	
Putting together \eqref{CA15}--\eqref{CA17} we get \eqref{CA7}. 

Finally, to see \eqref{CA88}, \eqref{CA8} we realize that the first term 
in \eqref{CA10} depending on $\Grad R_n(t,x)$ vanishes on the set 
$\cup_{n=1}^N B(r, \vc{h}_n (t))$ as $R_n(t,x) = 1$.

\end{proof}

\section{Convergence to a weak solution - proof of Theorem \ref{Tm1}}
\label{w}

We are ready to perform the limit in the sequence 
of solutions $(\vre, \vue)_{\ep > 0}$ to the fluid--structure 
interaction problem \eqref{m6}--\eqref{m9} 
under the hypotheses of Theorem \ref{Tm1}.

It follows from hypothesis \eqref{m11}, \eqref{m12}  that
\begin{equation} \label{w1}
\| \vr_{0, \ep} \|_{L^\gamma(\Td)} \aleq 1 \ \mbox{for a certain}\ \gamma > \frac{d}{2}. 
\end{equation}
meaning 
{
\begin{equation} \label{w2}
\intTd{ \vr_{0, \ep}^\gamma } \aleq 
\left(  \vr^\gamma_F +  N r^{d} {\Ov{\vr}}^\gamma_\ep \right) \aleq 1.
\end{equation}}
Consequently, as $\vr_\ep$ satisfy the equation of continuity 
\eqref{m6}, the $L^q-$norms are conserved and we deduce
\begin{equation} \label{w3}
\vre \to \vr_F \ \mbox{in}\ C([0,T]; L^q(\Td)) \ \mbox{for any}\ 1 \leq q < \gamma. 
\end{equation}
Note carefully that the total volume of the rigid balls is controlled by the quantity $N_\ep r_\ep^d \to 0$ as $\ep \to 0$.

Since  
\[
\intTd{ \vre \vc{g} \cdot \vue } \leq \| \vre \|_{L^1(\Td)}^{\frac{1}{2}} \|\vc{g} \|_{L^\infty(\Td; \R^d)} \| \sqrt{\vre} \vue \|_{L^2(\Td; \R^d)}, 
\]
and \eqref{m10} holds, 
the energy inequality \eqref{m9} yields the uniform bounds 
\begin{equation} \label{w3a}
(\vre \vue)_{\ep > 0} \ \mbox{bounded in}\ L^\infty(0,T; L^{\frac{2 \gamma}{\gamma + 1}}(\Td; \R^d)),  \ (\vue)_{\ep > 0} 
\ \mbox{bounded in}\ L^2(0,T; W^{1,2}(\Td; \R^d)).
\end{equation}
Thus, keeping \eqref{w3} in mind, we may extract a subsequence (not relabelled for simplicity) such that 
\begin{align} 
\vre \vue &\to \vr_F \vu \ \mbox{weakly-(*) in} \ L^\infty(0,T; 
L^{\frac{2 \gamma}{\gamma + 1}}(\Td, \R^d)), \ \frac{2 \gamma}{\gamma + 1} > \frac{6}{5} \ \mbox{if}\ d = 3, \ \gamma > 1 \ \mbox{if}\ d = 2. \br
\vue &\to \vu \ \mbox{weakly in}\ L^2(0,T; W^{1,2}(\Td; \R^d))
\ \mbox{as}\ \ep \to 0,\ \Div \vu = 0, \br 
\vre \vue \otimes \vue &\to \Ov{ \vr \vu \otimes \vu } \ \mbox{weakly in} \ L^2(0,T; L^r(\Td; R^{d^3})) 
\ \mbox{for some}\ r > 1. 
\label{w4}
\end{align}	

\subsection{Bounds on the rigid velocities}

To perform the limit in the momentum equation \eqref{m7}, we need uniform bounds on the rigid translations velocities 
$\frac{\D \hne }{\dt}$. Similarly to the above, they follow 
from the energy inequality \eqref{p10}. We report the following estimates shown in \cite[Sections 3.1, 3.2]{FeiRoyZar2022}:
\begin{align} 
\left| \frac{\D \hne (t) }{\dt} \right| &\leq c(p) \| \vue (t, \cdot) \|_{
W^{1,2}(\Td; \R^d) } r^{- \frac{1}{p}} \ \mbox{for any}\ n=1, \dots, N,\br  p &= 2 \ \mbox{if}\ d = 3,\ p \geq 1 \ \mbox{arbitrary}
\ \mbox{if}\ d = 2. 
\label{w5}
\end{align}	

\subsection{Convergence}

We have collected all the necessary material to show the convergence to the Navier--Stokes system claimed in Theorem \ref{Tm1}.
We consider the approximation $\bfphi_{\rm app}^\ep$ constructed in 
\eqref{A13}, \eqref{A14} for $r = r_\ep$, $N = N_\ep$.

\subsection{Convergence of terms containing space derivatives} 

In accordance with \eqref{CA5}, \eqref{w4}, we have 
\begin{equation} \label{conv1} 
\int_0^T \intTd{ \mathbb{S}(\Ds \vue) : \Grad \bap } \dt \to \int_0^T \intTd{ \mathbb{S}(\Ds \vu) : \Grad \bfphi } \dt \ \mbox{as}\ \ep \to 0
\end{equation}	
as long as 
\begin{equation} \label{conv2}
N^2_\ep |\log(r_\ep)|^2 N^{\frac{1}{2}}_\ep r^{\frac{d}{2}}_\ep \to 0 \ \mbox{as}\ \ep \to 0.	
\end{equation}	
In addition, seeing that 
\begin{align}
\int_0^T \intTd{ \vre (\vue \otimes \vue) : \Grad \bap } \dt &= \int_0^T \intTd{ \vre (\vue \otimes \vue) : \Ds \bap } \dt \br
&= \int_0^T \intTd{ \vr_F (\vue \otimes \vue) : \Ds \bap } \dt
\nonumber
\end{align}
we may infer that condition \eqref{conv2}, together with \eqref{w4}, guarantee the convergence in the convective term 
\begin{equation} \label{conv3}
\int_0^T \intTd{ \vre (\vue \otimes \vue) : \Grad \bap } \dt \to 
\int_0^T \intTd{ \Ov{\vr \vu \otimes \vu} : \Grad \bfphi } \dt \ \mbox{as}\ \ep \to 0.
\end{equation} 	

Finally, in view of \eqref{w1}--\eqref{w3}, we get 
\[
\int_0^T \intTd{ \vre \vc{g} \cdot \bap } \dt \to \int_0^T \intTd{ \vr_F \vc{g} \cdot \bfphi } \ \mbox{as}\ \ep \to 0
\] 
assuming again only \eqref{conv2}. 

\subsection{Convergence of time derivatives} 

Our next goal is to perform the limit in the integral 
\begin{align}
&\int_0^T \intTd{ \vre \vue \cdot \partial_t \bap } \dt \br &=
\int_0^T \int_{ \cup_{n=1}^{N_\ep} B(r_\ep, \hne)}  \vre \vue \cdot \partial_t \bap \ \dx \dt + 
\int_0^T \int_{\Td \setminus \cup_{n=1}^{N_\ep} B(r_\ep, \hne)}  \vr_F \vue \cdot \partial_t \bap \ \dx \dt, 
\label{conv5}
\end{align}
where, in accordance with \eqref{CA7}, \eqref{w4}, and the bounds of rigid velocities \eqref{w5},                                                  
\[
\int_0^T \int_{\Td \setminus \cup_{n=1}^{N_\ep} B(r_\ep, \hne)}  \vr_F \vue \cdot \partial_t \bap \ \dx \dt \to 
\int_0^T \intTd{  \vr_F \vu \cdot \partial_t \bfphi } \dt \ \mbox{as}\ \ep \to 0 
\]
as long as
\begin{equation} \label{conv4}
N^2_\ep |\log(r_\ep)|^2 N^{\frac{5}{6}}_\ep r_\ep^{2} \to 0 \ \mbox{if}\ d = 3,\ 
N^2_\ep |\log(r_\ep)|^2 N^{\frac{1}{q}}_\ep r^{\frac{1}{q}}_\ep \to 0,\ q > 1 \ \mbox{arbitrary} \ \mbox{if}\ d = 2.
\end{equation}
In both cases, \eqref{conv4} follows from \eqref{conv2}. 

As for the first integral on the right--hand side of \eqref{conv5}, we use the bound \eqref{CA8} 
to obtain 
\begin{align}
&\left| \int_0^T \int_{\cup_{n=1}^{N_\ep} B(r_\ep, \hne)}  \vre \vue \cdot \partial_t \bap \ \dx \dt \right| \leq
\Ov{\vr}_\ep \| \vue \|_{L^2(0, T; L^6(\Td; \R^d))} \| \partial_t \bap \|_{L^2(0,T; L^{\frac{6}{5}}(\cup_{n=1}^{N_\ep} B(r_\ep, \hne)))} \br 
&\aleq \Ov{\vr}_\ep N^{\frac{11}{6}}_\ep |\log(r_\ep)| r^2_\ep \to 0,  
\label{conv6} 
\end{align}
for $d = 3$, and 
\begin{align}
	&\left| \int_0^T \int_{\cup_{n=1}^{N_\ep} B(r_\ep, \hne)}  \vre \vue \cdot \partial_t \bap \ \dx \dt \right| \leq
	\Ov{\vr}_\ep \| \vue \|_{L^2(0, T; L^{q'}(\Td; \R^d))} \| \partial_t \bap \|_{L^2(0,T; L^{q}(\cup_{n=1}^{N_\ep} B(r_\ep, \hne)))} \br 
	&\aleq \Ov{\vr}_\ep N_\ep^{\frac{1}{q} + 1} |\log(r_\ep)| r^{\frac{1}{q}}_\ep \to 0,\ q > 1 \ \mbox{arbitrary}  
	\label{conv6a} 
\end{align}
for $d = 2$. Note that both convergences yield \eqref{w2}.

\subsection{Strong convergence of velocities}

To complete the proof of Theorem \ref{Tm1}, we have to show strong convergence of the velocities $\vue$. 
To this end, we use a variant of Aubin-Lions argument. 

First, it follows from the momentum equation \eqref{m7}, and the estimates obtained in the previous part that 
\begin{equation} \label{sc1}
\left[ t \mapsto \intTd{ \vre \vue \cdot \bap (t, \cdot) } \right] \to 
\left[ t \mapsto \intTd{ \vr_F \vu \cdot \bfphi (t, \cdot) } \right] \  \mbox{in}\ C[0,T]  
\end{equation}	
for any (smooth) $\bfphi$, $\Div \bfphi = 0$.
In view of \eqref{w4}, this yields 
\begin{equation} \label{sc2}
\int_0^T \intTd{ \vre \vue \cdot \bfphi } \dt \to 
\int_0^T \intTd{ \vr_F \vu \cdot \bfphi } \dt \ \mbox{for any}\ \bfphi \in L^2(0,T; W^{1,2}(\Td; \R^d))\ 
\ \Div \bfphi = 0.
\end{equation}	
Moreover, writing a general $\bfphi$ as its Helmholtz decomposition and using the convergences \eqref{w4}, \eqref{w4}, we can extend the validity 
of \eqref{sc2} to arbitrary $\bfphi$ (not necessarily solenoidal). 

Combining the first convergence in \eqref{w4} with \eqref{sc2} we conclude 
\[
\vre \vue (t, \cdot)  \to \vr_F \vu (t, \cdot) \ \mbox{in}\ L^{\frac{2 \gamma}{\gamma + 1}}(\Td; \R^d)-{\rm weak}\ \mbox{for a.a.}\ t \in (0,T) 
\] 
modulo a suitable subsequence. Finally, since $L^{\frac{2 \gamma}{\gamma + 1}}(\Td) \hookrightarrow \hookrightarrow W^{-1,2}(\Td)$, we obtain
\[
\vre \vue (t, \cdot)  \to \vr_F \vu (t, \cdot) \ \mbox{in}\ W^{-1,2}(\Td; \R^d) \ \mbox{for a.a.}\ t \in (0,T), 
\]
yielding the desired conclusion
\begin{align} 
\lim_{\ep \to 0} \int_0^T \intTd{ \vre |\vue|^2 } \dt &= 
\lim_{\ep \to 0} \int_0^T \intTd{ \vre \vue \cdot \vue } \dt = \intTd{ 
	\vr_F |\vu|^2 } \dt \dt 
\label{sc3}	
\end{align}	
which, together with the weak convergence established in \eqref{w4}, yields the desired strong convergence in the $L^2-$norm claimed in Theorem \ref{Tm1}. 
In particular, $\Ov{\vr \vu \otimes \vu} = \vr_F \vu \otimes \vu$.
We have completed the proof of Theorem \ref{Tm1}. 

\section{Convergence to a strong solution - proof of Theorem \ref{Tm3}}
\label{T}

In accordance with the hypotheses of Theorem \ref{Tm3}, we suppose the limit Navier--Stokes system admits a classical (smooth) solution $\vu$. 
Our goal is to estimates the time increments of the relative energy, specifically, 
\begin{align}
\left[ \intTd{ \frac{1}{2} \vre |\vue - \vu^{\ep}_{\rm app}|^2 } \right]_{t = 0}^{t = \tau} &=  
\left[ \intTd{ \frac{1}{2} \vre |\vue|^2 } \right]_{t = 0}^{t = \tau} + \left[ \intTd{ \frac{1}{2} \vre |\vu^{\ep}_{\rm app}|^2 } \right]_{t = 0}^{t = \tau} \br 
- \left[ \intTd{  \vre \vue \cdot \vu^{\ep}_{\rm app} } \right]_{t = 0}^{t = \tau}. 
\label{T1}
\end{align}
This quantity may be evaluated using the energy inequality \eqref{m9}, together with the equation of continuity \eqref{m6} and the momentum equation 
\eqref{m7}. Indeed the quantities $|\vu_{\rm app}^\ep|^2$ and $\vu^\ep_{\rm app}$ are eligible test functions in 
the equation of continuity \eqref{m6} and the momentum balance \eqref{m7}, respectively.
Recall that $\vre \vue$ belong to the space \eqref{w3a} and as $\gamma>\frac{3}{2}$ implies $\frac{2\gamma}{\gamma+1}>\frac{6}{5}$, we have specifically 
\begin{align} 
(\vre \vue)_{\ep > 0} \ &\mbox{is bounded in}\ L^\infty(0,T; L^{\Gamma}(\Td; \R^d)),\ \Gamma > \frac{6}{5} \ \mbox{if}\ d = 3,\ \Gamma > 1 \ \mbox{if}\ d = 2, \br   \ (\vue)_{\ep > 0} 
\ &\mbox{bounded in}\ L^2(0,T; W^{1,2}(\Td; \R^d)).
 \label{T2}
\end{align}

After a straightforward manipulation, we get the relative energy inequality of the form
\begin{align} 
\frac{1}{2} &\intTd{ \vre |\vue - \vu^\ep_{\rm app}|^2 (\tau, \cdot) } + 
\int_0^\tau \intTd{ \Big( \mathbb{S}(\Ds \vue) -  \mathbb{S}(\Ds  	\vu^\ep_{\rm app}) \Big) : \Big( \Ds \vue - \Ds  	\vu^\ep_{\rm app} \Big) } \dt \br 
&\leq \frac{1}{2} \intTd{ \vr_{0, \ep} |\vu_{0, \ep} - \vu_{\rm app}(0, \cdot) |^2 } -
\int_0^\tau \intTd{ \vre \Big[ (\vue - \au) \otimes (\vue - \au) \Big] : \Grad \au } \dt \br 
&+ \int_0^\tau \intTd{ \vre (\partial_t \au + \au \cdot \Grad \au ) \cdot (\vue - \au) } \dt \br 
&+ \int_0^\tau \intTd{ \mathbb{S} (\Ds \au) : \Ds (\vue - \au) } \dt 
- \int_0^\tau \intTd{ \vre \vc{g} \cdot (\vue - \au) } \dt\ \mbox{for a.a.}\ \tau \in (0,T).
\label{T3}	
\end{align}
Here, it is important to notice that 
\begin{align}
\int_0^\tau &\intTd{ \vre \Big[ (\vue - \au) \otimes (\vue - \au) \Big] : \Grad \au } \dt \br &= 
\int_0^\tau \intTd{ \vr_F \Big[ (\vue - \au) \otimes (\vue - \au) \Big] : \Ds \au } \dt \br &=
\int_0^\tau \intTd{ \vr_F \Big[ (\vue - \au) \otimes (\vue - \au) \Big] : ( \Ds \au - \Ds \vu) } \dt \br &+
\int_0^\tau \intTd{ \vr_F \Big[ (\vue - \au) \otimes (\vue - \au) \Big] : \Ds \vu  } 
\label{T4}
\end{align}
Consequently, we may rewrite \eqref{T3} as 
\begin{align} 
	\frac{1}{2} &\intTd{ \vre |\vue - \vu^\ep_{\rm app}|^2 (\tau, \cdot) } + 
	\int_0^\tau \intTd{ \Big( \mathbb{S}(\Ds \vue) -  \mathbb{S}(\Ds  	\vu^\ep_{\rm app}) \Big) : \Big( \Ds \vue - \Ds  	\vu^\ep_{\rm app} \Big) } \dt \br 
	&\leq \frac{1}{2} \intTd{ \vr_{0, \ep} |\vu_{0, \ep} - \vu_{\rm app}^\ep(0, \cdot) |^2 } - 
	\int_0^\tau \intTd{ \vre \Big[ (\vue - \au) \otimes (\vue - \au) \Big] : \Ds \vu } \dt \br 
	&+ \int_0^\tau \intTd{ \vr_F (\partial_t \au + \au \cdot \Grad \au ) \cdot (\vue - \au) } \dt \br 
	&+ \int_0^\tau \intTd{ \mathbb{S} (\Ds \au) : \Ds (\vue - \au) } \dt 
	- \int_0^\tau \intTd{ \vre \vc{g} \cdot (\vue - \au) } \dt\br 
	&+ \int_0^\tau \intTd{ \vr_F \Big[ (\vue - \au) \otimes (\vue - \au) \Big] : ( \Ds \au - \Ds \vu) } \dt \br
	&+ \int_0^\tau \intTd{ (\vr_F - \vre) \Big[ (\vue - \au) \otimes (\vue - \au) \Big] : \Ds \vu  } \dt \br
	&+ \int_0^\tau \intTd{ (\vre - \vr_F) (\partial_t \au + \au \cdot \Grad \au ) \cdot (\vue - \au) } \dt \ 
	 \mbox{for a.a.}\ \tau \in (0,T).
	\label{T5}	
\end{align}

Finally, we use the fact that $\vu$ is a smooth solution to the Navier-Stokes system, 
\begin{align} 
	\frac{1}{2} &\intTd{ \vre |\vue - \vu^\ep_{\rm app}|^2 (\tau, \cdot) } + 
	\int_0^\tau \intTd{ \Big( \mathbb{S}(\Ds \vue) -  \mathbb{S}(\Ds  	\vu^\ep_{\rm app}) \Big) : \Big( \Ds \vue - \Ds  	\vu^\ep_{\rm app} \Big) } \dt \br 
	&\leq \frac{1}{2} \intTd{ \vr_{0, \ep} |\vu_{0, \ep} - \vu_{\rm app}^\ep(0, \cdot) |^2 } - 
	\int_0^\tau \intTd{ \vr_F \Big[ (\vue - \au) \otimes (\vue - \au) \Big] : \Ds \vu } \dt \br 
	&+ \int_0^\tau \intTd{ \vr_F \Big(\partial_t \au + \au \cdot \Grad \au - \partial_t \vu - \vu \cdot \Grad \vu \Big) \cdot (\vue - \au) } \dt \br 
	&+ \int_0^\tau \intTd{ \Big( \mathbb{S} (\Ds \au) - \mathbb{S}(\Ds \vu) \Big) : \Ds (\vue - \au) } \dt 
	\br 
	&+ \int_0^\tau \intTd{ \vr_F \Big[ (\vue - \au) \otimes (\vue - \au) \Big] : ( \Ds \au - \Ds \vu) } \dt \br
	&- \int_0^\tau \intTd{ (\vre - \vr_F) \vc{g} \cdot (\vue - \au) } \dt \br
    &+ \int_0^\tau \intTd{ (\vr_F - \vre) \Big[ (\vue - \au) \otimes (\vue - \au) \Big] : \Ds \vu  } \dt\br
	&+ \int_0^\tau \intTd{ (\vre - \vr_F) (\partial_t \au + \au \cdot \Grad \au ) \cdot (\vue - \au) } \dt \ 
	\mbox{for a.a.}\ \tau \in (0,T).
	\label{T6}	
\end{align}

We proceed in several steps:

\medskip 

\noindent{\bf Step 1.}

\begin{align} 
\int_0^\tau &\intTd{ \vr_F \Big(\partial_t \au - \partial_t \vu \Big) \cdot (\vue - \au) } \dt \br
&\leq \omega \int_0^\tau \| \vue - \au \|^2_{W^{1,2}(\Td; \R^d)} \dt + c(\omega) \int_0^\tau  \left\|\partial_t \au - \partial_t \vu  \right\|^2_{L^q(\Td; \R^d)}, 
\nonumber	
\end{align}
where $q = \frac{6}{5} \ \mbox{if}\ d = 3,\ q > 1 \ \mbox{if} \ d = 2$,
for any $\omega > 0$.
It follows from the bounds \eqref{CA7}, \eqref{w5} that  
\begin{align} 
\int_0^T \| \partial_t \au - \partial_t \vu \|^2_{L^{\frac{6}{5}}(\Tdd; R^3)} \dt &\aleq |\log(r_\ep)|^4 N^{\frac{17}{3}}_\ep r^{4}_\ep \ \mbox{if}\ d = 3, \br
\int_0^T \| \partial_t \au - \partial_t \vu \|^2_{L^{q}(\Tdd; R^3)} & \aleq |\log(r_\ep)|^4 N^{4 + \frac{2}{q}}_\ep r^{\frac{2}{q}}_\ep, \ q > 1,
\ \mbox{if}\ d = 2.  
\label{T7} 
\end{align}

\noindent{\bf Step 2.}

\begin{align} 
\int_0^\tau &\intTd{ \Big(\au \cdot \Grad \au - \vu \cdot \Grad \vu \Big) \cdot ( \vue - \au ) } \dt \br
&= -	\int_0^\tau \intTd{ \Big(\au \otimes \au - \vu \otimes \vu \Big) : \Grad ( \vue - \au ) } \dt \br 
&\leq \omega \int_0^\tau \intTd{ \| \Grad ( \vue - \au ) \|^2_{L^2(\Td; R^{d^2})} } \dt \br &+ 
	c(\omega) \left(
\int_0^\tau \intTd{ \| \vu \|_{L^4(\Td; \R^d)}^2 \| \au - \vu \|^2_{L^4(\Td; \R^d)}   } \dt \right),
\nonumber
\end{align}	
where, by virtue of \eqref{CA3}, 
\begin{equation} \label{T8}
\sup_{t \in [0,T]} \| \au - \vu \|_{L^4(\Td; \R^d)} \aleq N^{\frac{9}{4}}_\ep r^{1 + \frac{d}{4}}_\ep |\log(r_\ep)|^2.
\end{equation}
\noindent{\bf Step 3.}

\begin{align} 
&\left| \int_0^\tau \intTd{ \Big( \mathbb{S} (\Ds \au) - \mathbb{S}(\Ds \vu) \Big) : \Ds (\vue - \au) } \dt \right| \br 
&\leq \omega \int_0^\tau \|\Ds \vue - \Ds \au \|^2_{L^2(\Td; \R^d)} \dt + c(\omega)  
\int_0^\tau \| \Grad \au - \Grad \vu \|^2_{L^2(\Td; R^{d^2})} \dt, 
\nonumber
\end{align}	
where, by virtue of \eqref{CA5}
\begin{equation} \label{T9}
\sup_{t \in [0,T]}  \| \Grad \au - \Grad \vu \|_{L^2(\Td; R^{d^2})} \aleq N^{\frac{5}{2}}_\ep |\log(r_\ep)|^2 r^{\frac{d}{2}}_\ep.
\end{equation}

\noindent{\bf Step 4.}

\medskip
\noindent
Similarly to {\bf Step 2}, 
\begin{align}
\int_0^\tau &\intTd{ \vr_F \Big[ (\vue - \au) \otimes (\vue - \au) \Big] : ( \Ds \au - \Ds \vu) } \dt \br &\leq 
\omega \int_0^\tau \intTd{ \| \Grad ( \vue - \au ) \|^2_{L^2(\Td; R^{d^2})} } \dt + 
c(\omega) \int_0^\tau \intTd{ \| \au - \vu \|^4_{L^4(\Td; \R^d)}   } \dt.
\noindent
\end{align}

\medskip 

\noindent{\bf Step 5.}

\begin{align}
\int_0^\tau &\intTd{ (\vre-\vr_F) \Big[ (\vue - \au) \otimes (\vue - \au) \Big] : \Ds \vu  } \dt \br &= 
\int_0^\tau \int_{\cup_{n=1}^{N_\ep} B(r_\ep, \vc{h}_n)} (\vr^\ep_S - \vr_F) \Big[ (\vue - \au) \otimes (\vue - \au) \Big] : \Ds \vu \dx  \dt
\nonumber
\end{align}
Consequently, 
\begin{align}
&\left| \int_0^\tau \int_{\cup_{n=1}^{N_\ep} B(r_\ep, \vc{h}_n)} (\vr^\ep_S - \vr_F) \Big[ (\vue - \au) \otimes (\vue - \au) \Big] : \Ds \vu \dx  \dt	\right| \br
&\leq \|\vue - \au\|^2_{L^2(0,T; W^{1,2}(\Td))}|\Ov{\vr}_\ep||\cup_{n=1}^{N_\ep} B(r_\ep, \vc{h}_n)|^{\frac{1}{q}},
\nonumber
\end{align}
where $q = \frac{3}{2} \ \mbox{if}\ d = 3,\ q > 1 \ \mbox{if} \ d = 2$. Under the hypotheses of Theorem \ref{Tm3}, $|\Ov{\vr}_\ep|\cup_{n=1}^{N_\ep} B(r_\ep, \vc{h}_n)|^{\frac{1}{q}} \rightarrow 0$ as $\ep\rightarrow 0$.  
\medskip 

\noindent{\bf Step 6.}

\begin{align}
\int_0^\tau &\intTd{ (\vre - \vr_F) (\partial_t \au + \au \cdot \Grad \au ) \cdot (\vue - \au) } \dt \br &= 
\int_0^\tau \int_{\cup_{n=1}^{N_\ep} B(r_\ep, \vc{h}_n)} (\vr^\ep_S - \vr_F) \partial_t \au \cdot (\vue - \au) \dx \dt
\nonumber
\end{align}
Consequently, 
\begin{align}
&\left| \int_0^\tau \int_{\cup_{n=1}^{N_\ep} B(r_\ep, \vc{h}_n)} (\vr^\ep_S - \vr_F) \partial_t \au \cdot (\vue - \au) \dx \dt	\right| \br
&\leq \omega \int_0^T \| \vue - \au \|_{W^{1,2}(\Td; \R^d)}^2 \dt + 
c(\omega) \Ov{\vr}_\ep^2 \int_0^\tau \|  \partial_t \au  \|_{L^q( \cup_{n=1}^{N_\ep} B(r_\ep, \vc{h}_n) )}^2 \dt,\br q &= \frac{6}{5} \ \mbox{if}\ d = 3, 
\ q > 1 \mbox{ if}\ d = 2, 
\nonumber
\end{align}	
where by virtue of \eqref{CA8}, \eqref{w5},
\begin{align} 
\int_0^\tau \|  \partial_t \au  \|_{L^q( \cup_{n=1}^{N_\ep} B(r_\ep, \vc{h}_n) )}^2 \dt \aleq N_\ep^{\frac{11}{3}} r_\ep^{4}  |\log(r_\ep)|^2 \ \mbox{if}\ d = 3, \br 
\int_0^\tau \|  \partial_t \au  \|_{L^q( \cup_{n=1}^{N_\ep} B(r_\ep, \vc{h}_n) )}^2 \dt \aleq N_\ep^{2 + \frac{2}{q}} r_\ep^{\frac{2}{q}}  |\log(r_\ep)|^2 \ \mbox{if}\ d = 2.
\label{T10}
\end{align}	
The integral containing $\vc{g}$ can be handled in a similar manner. 

Finally, fixing $\omega > 0$ small enough and evoking the hypotheses of Theorem \ref{Tm3}, we may  
apply Gr\" onwall's type argument to the relative energy inequality \eqref{T6} to obtain
\begin{align}
\sup_{t \in [0,T]} & \intTd{ \vre |\vue - \vu^\ep_{\rm app}|^2 (\tau, \cdot) } + 
\int_0^T \intTd{ \Big( \mathbb{S}(\Ds \vue) -  \mathbb{S}(\Ds  	\vu^\ep_{\rm app}) \Big) : \Big( \Ds \vue - \Ds  	\vu^\ep_{\rm app} \Big) } \dt \br
&\aleq r_\ep^\alpha  \beta(\ep),\ \mbox{where} 
\ \beta(\ep) \to 0 \ \mbox{as}\ \ep \to 0,\ \alpha = 1 \ \mbox{if}\ d = 3,\ \alpha > 0 \ \mbox{if}\ d = 2.
\label{T11}
\end{align}
Next, similarly to \eqref{w5}, we deduce the estimates 
\begin{align} 
r_\ep^\alpha \left| \frac{ \D \hnet }{\dt} - \vu(t, X_{B, 4 r_\ep}(\hnet)  \right|^2
\aleq \intTd{ \Big( \mathbb{S}(\Ds \vue) -  \mathbb{S}(\Ds  	\vu^\ep_{\rm app}) \Big) : \Big( \Ds \vue - \Ds  	\vu^\ep_{\rm app} \Big) }
\nonumber
\end{align}	
for any $n=1, \dots, N$, where $\alpha$ is the same as in \eqref{T11}. Thus we conclude 
\begin{equation} \label{T12}
\int_0^T \left| \frac{ \D \hnet }{\dt} - \vu(t, X_{B, 4 r_\ep}(\hnet)  \right|^2	\dt \to 0 \ \mbox{as}\ \ep \to 0 
\ \mbox{uniformly for}\ n=1,\dots, N, 
\end{equation}	
which, together with \eqref{F5c}, completes the proof of Theorem \ref{Tm3}.

\def\cprime{$'$} \def\ocirc#1{\ifmmode\setbox0=\hbox{$#1$}\dimen0=\ht0
	\advance\dimen0 by1pt\rlap{\hbox to\wd0{\hss\raise\dimen0
			\hbox{\hskip.2em$\scriptscriptstyle\circ$}\hss}}#1\else {\accent"17 #1}\fi}


\end{document}